\documentclass{article}

\usepackage[utf8]{inputenc}
\usepackage[english]{babel}

\usepackage{amsmath}
\usepackage{amssymb}
\usepackage{amsthm}
\usepackage{mathtools}
\usepackage{amsfonts}
\usepackage{enumerate}
\usepackage{enumitem}

\usepackage{hyperref}

\usepackage{tikz}
\usetikzlibrary{cd}

\mathtoolsset{showonlyrefs}
\allowdisplaybreaks

\newcommand{\N}{\mathbb{N}}
\newcommand{\Z}{\mathbb{Z}}
\newcommand{\R}{\mathbb{R}}
\newcommand{\C}{\mathbb{C}}

\newcommand{\Pj}{\mathbb{P}}
\newcommand{\Hyp}{\mathcal{H}}

\newcommand{\U}{\mathrm{U}}
\newcommand{\SU}{\mathrm{SU}}

\newcommand{\SP}{\mathrm{Sp}}

\newcommand{\spc}{\mathrm{S}}

\newcommand{\Hom}{\mathrm{Hom}}

\newcommand{\Ric}{\mathrm{Ric}}
\newcommand{\scal}{\mathrm{scal}}

\newcommand{\smooth}[1]{\mathcal{C}^{\infty}{\left( {#1} \right)}}
\newcommand{\field}[1]{\mathfrak{X}{\left({{#1}} \right)}}
\newcommand{\spn}[1]{\left\langle {#1} \right\rangle}
\newcommand{\id}{\mathrm{id}}

\newcommand{\Ad}{\mathrm{Ad}}
\newcommand{\lie}[1]{\mathfrak{#1}}
\newcommand{\Lie}[1]{\mathcal{L}_{#1}}

\newcommand{\tr}{\mathrm{tr}}

\newcommand{\rvline}{\hspace*{-\arraycolsep}\vline\hspace*{-\arraycolsep}}

\renewcommand{\Re}{\operatorname{Re}}
\renewcommand{\Im}{\operatorname{Im}}

\newcommand{\RE}{\mathfrak{R}}

\newcommand{\norm}[1]{\left\lVert{#1}\right\rVert}

\newtheorem{defi}{Definition}[section]
\newtheorem{theo}[defi]{Theorem}
\newtheorem{prop}[defi]{Proposition}
\newtheorem{cor}[defi]{Corollary}
\newtheorem{lemma}[defi]{Lemma}
\newtheorem{rmk}[defi]{Remark}

\title{Construction of projective special K\"{a}hler manifolds}

\author{Mauro Mantegazza\footnote{Dipartimento di matematica e applicazioni, Università degli studi Milano-Bicocca, via Roberto Cozzi, 55, Bldg.U5, 20125, Milano, Italy. E-mail: \href{mailto:mauro.mantegazza.uni@gmail.com}{mauro.mantegazza.uni@gmail.com}, ORCID: 0000-0001-7938-2950}}

\begin{document}
\maketitle
\begin{abstract}
In this paper we present an intrinsic characterisation of projective special K\"{a}hler manifolds in terms of a symmetric tensor satisfying certain differential and algebraic conditions.
We show that this tensor vanishes precisely when the structure is locally isomorphic to a standard projective special K\"{a}hler structure on $\SU(n,1)/\spc(\U(n)\U(1))$.
We use this characterisation to classify $4$-dimensional projective special K\"{a}hler Lie groups.
\renewcommand{\thefootnote}{\fnsymbol{footnote}}
\begin{NoHyper}
\footnotetext{\emph{MSC class}: 53C55; 53C26, 22E25, 53C80.}
\end{NoHyper}
\renewcommand{\thefootnote}{\arabic{footnote}}
\end{abstract}
\section{Introduction}
Projective special K\"{a}hler manifolds are a special class of K\"{a}hler quotients of conic special K\"{a}hler manifolds which is a class of pseudo-K\"{a}hler manifolds endowed with a symplectic, flat, torsion-free connection and an infinitesimal homothety.

Explicit examples can be found in \cite{CortesClassHomoSS}, where homogeneous projective special K\"{a}hler manifolds of semisimple Lie groups are classified.
A notable case appearing in this list is the complex hyperbolic $n$-space.
Many projective special K\"{a}hler manifolds can be constructed via the so called r-map \cite{CiSC}, which is a construction arising from supergravity and string theory allowing to build a projective special K\"{a}hler manifold starting from a homogeneous cubic polynomial.
See \cite{CortesClassPSKr} for a classification of $6$-dimensional manifolds that can be constructed via the r-map.
Another example is obtained by taking the Weil-Petersson metric on the space of complex structure deformations on a Calabi-Yau $3$-dimensional manifold \cite{CortesHK1998}.

Projective special K\"{a}hler manifolds appear in the study of supergravity and mirror symmetry with the name \emph{local special K\"{a}hler manifolds} (see \cite{Fre1995} and \cite{Freed1999} for more details on their history and applications to physics, and in particular \cite{CandelasOssa1991} for their importance in mirror symmetry).
The name projective special K\"{a}hler was given by Freed in \cite{Freed1999} where he also shows how such manifolds are quotients of special K\"{a}hler ones \cite[Proposition 4.6, p.\ 20]{Freed1999} (see e.g.\ \cite{SpComMan} for the relation between this definition and the one we will use in this work).

Projective special K\"{a}hler manifolds are not only interesting on their own, as they find an important application in quaternion K\"{a}hler geometry.
The construction known as c-map, also arising from the same areas of physics, allows in fact to create quaternion K\"{a}hler manifolds of negative scalar curvature starting from projective special K\"{a}hler ones \cite{CiSC}, \cite{Conification2013}, \cite{CortesHKQK}, \cite{Swann2015}, \cite{HitHKQK}, \cite{FerSab}, \cite{CFG1989}.
Quaternion K\"{a}hler manifolds are orientable smooth Riemannian manifolds of dimension $4n$ with $n\ge 2$, whose holonomy group is a subgroup of $\SP(n)\SP(1)$ not contained in $\SP(n)$.
They are important since they are a special family of Einstein manifolds with non vanishing Ricci tensor, corresponding to one of the possible holonomy groups of a locally irreducible, non-locally symmetric, simply connected Riemannian manifold in Berger's list (see \cite{Berger}).

In this paper we present a characterisation of projective special K\"{a}hler manifolds that will hopefully shed more light on this type of structure.
Our characterisation is intrinsic in the sense that we reduce the projective special K\"{a}hler structure to data solely defined on the manifold itself.
The characterisation is obtained by means of a locally defined symmetric tensor that we call deviance, satisfying certain conditions: a differential one and an algebraic one.
The deviance tensor emerges from the difference between two naturally occurring connections on the conic special K\"{a}hler manifold over a projective special K\"{a}hler one.
Our description arises by writing the two connections with respect to a local frame, and relating the difference to a tensor defined on the basis.
We also prove a lower bound for the scalar curvature, which is reached exactly when the deviance is zero; this condition characterises projective special K\"{a}hler manifolds isomorphic to the complex hyperbolic $n$-space if one assumes the manifold complete, connected and simply connected.
For projective special K\"ahler manifolds of elliptic type, i.e.\ such that the corresponding conic special K\"ahler metric is positive definite, we have a similar result in \cite[Theorem 16, p.\ 126]{BC2003}, where complex projective spaces are characterised as the only complete projective special K\"ahler domains of elliptic type.
In that context, in fact, the vanishing of the deviance is induced by the completeness condition.

Our characterisation provides a simpler way to construct projective special K\"{a}hler manifolds, and we display this by classifying all possible projective special K\"{a}hler structures on $4$-dimensional Lie groups.
Our classification relies on the classification of K\"{a}hler Lie groups by Ovando \cite{Ovando2004}.

We note that an intrinsic characterisation of projective special K\"{a}hler Lie groups has been obtained independently in a very recent paper by Macia and Swann \cite{MaciaSwann2019}.
Our setting is slightly more general, since in our study of projective special K\"{a}hler Lie groups we do not assume the deviance tensor to be left-invariant.
In \cite{MaciaSwann2019} it is also shown that projective special K\"{a}hler Lie groups determine quaternion K\"{a}hler Lie groups via the c-map, if one assumes the exactness of the K\"{a}hler form and the invariance of the flat connection.
A similar result, holding in the case that the projective special K\"{a}hler Lie group is the quotient of an affine special K\"{a}hler domain, can be deduced from the more general result \cite[Corollary 24, p.\ 33]{CortesCohomogeneityOne}.

Since we are ultimately interested in the c-map, throughout this paper we adopt the same convention as \cite{CiSC}, where we only consider projective special K\"{a}hler manifolds obtained from conic special K\"{a}hler manifolds with signature $(2n,2)$.
Nonetheless, our characterisation can be generalised to generic signatures.
It is worth mentioning that the deviance, being a symmetric tensor of type (3,0), can often be seen as a homogeneous polynomial of degree three, which may have a role in providing a partial inversion to the r-map.

We also use our characterisation to show that, on a K\"{a}hler manifold, the existence of a symmetric tensor satisfying the deviance conditions implies the existence of a whole family of projective special K\"{a}hler structures and we provide sufficient conditions for said structures to be isomorphic.

\textbf{Acknowledgements}. This paper is part of the author's PhD thesis \cite{PhDThesis} written under the supervision of Diego Conti.
Part of the study leading to this work was carried out during a visiting period at QGM, Aarhus; the author wishes to thank Andrew Swann for many useful discussions during that visit.
This is a preprint of an article published in \emph{Annali di Matematica Pura ed Applicata (1923 -)}.
The final authenticated version is available online at: \href{https://doi.org/10.1007/s10231-021-01096-4}{https://doi.org/10.1007/s10231-021-01096-4}.

\section{Definitions}
In this section we are introducing the basic objects that we are going to discuss in this work.

The coming definition involves a flat connection $\nabla$ and its exterior covariant derivative operator $d^{\nabla}$.
\begin{defi}\label{def:CSK}
A \emph{conic special K\"{a}hler} manifold is the data of a pseudo-K\"{a}hler manifold $(\widetilde{M},\widetilde{g},\widetilde{I},\widetilde{\omega})$ with a flat, torsion-free, symplectic connection $\nabla$ and a vector field $\xi$ such that
\begin{enumerate}
\item $d^{\nabla}\widetilde{I}=0$ where we interpret $\widetilde{I}$ as a $1$-form with values in $T\widetilde{M}$;
\item $\widetilde{g}(\xi,\xi)$ is nowhere vanishing;
\item $\nabla \xi=\widetilde{\nabla}^{LC} \xi=\id$;
\item\label{cond:segnaturaPSK} $\widetilde{g}$ is negative definite on $\spn{\xi,I\xi}$ and positive definite on its orthogonal complement.
\end{enumerate}
Here $\widetilde{\nabla}^{LC}$ is the Levi-Civita connection.
\end{defi}
We will adopt the convention $\widetilde{\omega}=\widetilde{g}(\widetilde{I}\cdot,\cdot)$.
Definition \ref{def:CSK} is identical to Definition 3 in \cite{CiSC} if we take $-g$ as metric.
 
We start by showing how the Lie derivative along $\xi$ and $I\xi$ in a conic special K\"{a}hler manifold behaves on the K\"{a}hler structure.
\begin{lemma}[Lemma 3.2, p.\ 1336 in \cite{Swann2015}]\label{lemma:azioneInfinitesimaXieIxi}
Let $(\widetilde{M},\widetilde{g},\widetilde{I},\widetilde{\omega},\nabla,\xi)$ be a conic special K\"{a}hler manifold, then:
\begin{enumerate}
\item $\xi$ is a homothety of scaling factor $2$ preserving $\widetilde{I}$;
\item $\widetilde{I}\xi$ preserves the K\"{a}hler structure.
\end{enumerate}
\end{lemma}
\begin{proof}
See e.g.\ \cite{Swann2015} where $X=-I\xi$.
\end{proof}
Before proceeding, we write the following lemma for future reference.
\begin{lemma}\label{lemma:XeXi}
In a conic special K\"{a}hler manifold $(\widetilde{M},\widetilde{g},\widetilde{I},\widetilde{\omega},\nabla,\xi)$, $\nabla (\widetilde{I}\xi)=\widetilde{I}$.
\end{lemma}
\begin{proof}
For all $X\in\field{\widetilde{M}}$,
\begin{align*}
\nabla_X (\widetilde{I}\xi)-\widetilde{I}X
&=(\nabla_X \widetilde{I})\xi+\widetilde{I}\nabla_X\xi-\widetilde{I}X
=(\nabla_X \widetilde{I})\xi
=(\nabla_{\xi} \widetilde{I})X\\
&=\nabla_{\xi} (\widetilde{I}X)-\widetilde{I}\nabla_{\xi} X
=\nabla_{\widetilde{I}X} (\xi)+[\xi,\widetilde{I}X]-\widetilde{I}\left(\nabla_{X} \xi+[\xi,X]\right)\\
&=\widetilde{I}X+\Lie{\xi} (\widetilde{I}X)-\widetilde{I}X-\widetilde{I}\Lie{\xi} X
=(\Lie{\xi} \widetilde{I})X
=0.\qedhere
\end{align*}
\end{proof}
If we compare Definition \ref{def:CSK} with Definition 3.1 in \cite{Swann2015}, we notice that the main difference is the signature of the metric: it is enough to add condition \ref{cond:segnaturaPSK} to the latter and to define $X=-I\xi$ in order to obtain two equivalent definitions.
The proof of the equivalence is obtained by Lemma \ref{lemma:XeXi}.

\begin{defi}\label{def:PSK}
A \emph{projective special K\"{a}hler} manifold is a K\"{a}hler manifold $M$ endowed with a $\C^*$-bundle $\pi\colon \widetilde{M}\to M$ with $(\widetilde{M},\widetilde{g},\widetilde{I},\widetilde{\omega},\nabla,\xi)$ conic special K\"{a}hler such that $\xi$ and $I\xi$ are the fundamental vector fields associated to $1,i\in\C$ respectively and $M$ is the K\"{a}hler quotient with respect to the induced $\U(1)$-action.
In this case we say that $M$ has a projective special K\"{a}hler structure.

For brevity, we will often denote a projective special K\"{a}hler manifold by $(\pi\colon\widetilde{M}\to M,\nabla)$.
\end{defi}
\begin{rmk}
We shall see later that by construction, the action is always Hamiltonian with moment map $-\widetilde{g}(\xi,\xi)$, and the choice of the level set affects the quotient only up to scaling.
\end{rmk}
Concerning the notation for projective special K\"{a}hler manifolds as in Definition \ref{def:PSK}, when a tensor or a connection is possessed by both $\widetilde{M}$ and $M$, we will write them and everything concerning them (torsion, curvature forms, covariant exterior differentials) on $\widetilde{M}$ with $\widetilde{(\cdot)}$ above, whereas the corresponding objects on $M$ will be denoted without it.

\section{Difference tensor}\label{sec:differenceTensor}
This section is devoted to the tensor obtained as difference between the flat and Levi-Civita connection on a conic special K\"{a}hler manifold.
We present the known symmetry of this tensor and write the flatness condition in terms of it \cite[p.\ 9-11]{Freed1999}.

Before talking about the difference tensor, we will introduce some notation and definitions.
Following \cite{RedBook}, if $V$ is a complex representation with a real structure $\sigma$, we define
\begin{equation}
[V]:=\{v\in V|\sigma(v)=v\}.
\end{equation}
Otherwise, for any complex representation $V$,
\begin{equation}
[\![V]\!]:=[V\oplus \overline{V}]
\end{equation}
where $\overline{V}$ is the conjugate representation of $V$.
In particular, the following complex Lie algebra isomorphisms hold:
\begin{equation}
[V]\otimes_{\R} \C\cong V,\qquad
[\![V]\!]\otimes_{\R} \C\cong V\oplus \overline{V}.
\end{equation}
The same notation is used for the associated vector bundles.

Given an almost complex manifold $(M,I)$, let $T_{1,0}M$ be the holomorphic cotangent bundle.
For all $p\in\N$, we denote its $p$-th symmetric power by $S_{p,0}M$.

Given a (pseudo-)Riemannian manifold $(M,g)$, we denote by $\flat$ and $\sharp$ the musical isomorphisms induced by $g$, and we can define the following isomorphism 
\begin{equation}\label{eq:bemolle}
\flat_2=\id\otimes \flat\otimes \id\colon T^*M\otimes TM\otimes T^*M\to T_{3}M
\end{equation}
with inverse $\sharp_2:=\id\otimes \sharp\otimes \id$.

Returning now to the main topic of this section, let $(\widetilde{M},\widetilde{g},\widetilde{I},\widetilde{\omega},\nabla,\xi)$ be a conic special K\"{a}hler manifold of dimension $n+1$.
We define $\widetilde{\eta}$ as the (1,2)-tensor such that for all vector fields $X$, $Y$ on $\widetilde{M}$ we have $\widetilde{\eta}_X Y=\nabla_X Y-\widetilde{\nabla}^{LC}_X Y$, where the employed notation $\widetilde{\eta}_X Y$ means $\widetilde{\eta}(X,Y)$.

Consider frames adapted to the pseudo-K\"{a}hler structure, hence such that the linear model is $(\R^{2n+2},g_0,I_0,\omega_0)$, where $g_0=\sum_{k=1}^{2k}(e^k)^2-(e^{2n+1})^2-(e^{2n+2})^2$, $I_0 e_{2k-1}=e_{2k}$ for $k=1,\dots,n+1$ and $\omega_0=g_0(I_0\cdot,\cdot)$.
Let $\omega^{\nabla}$ and $\widetilde{\omega}^{LC}$ be the connection forms corresponding respectively to the flat and the Levi-Civita connections represented with respect to an adapted frame.
Thus we have
\begin{equation}
\omega^{\nabla}=\widetilde{\omega}^{LC}+\widetilde{\eta}.
\end{equation}

We know that by the theory of Hessian manifolds, $\widetilde{\eta}$ is symmetric (see e.g.\ \cite[\S 3, p.\ 194]{ShimaV} or \cite[\S 1, p.\ 736]{ShimaCSC}). More precisely, we recall the following result (see \cite[Proposition 1.34, p.\ 39]{Freed1999}, or \cite[Proposition 4, p.\ 1743]{BauesCortes} and \cite[Lemma 3, p.\ 1745]{BauesCortes}).
\begin{lemma}\label{lemma:eta_Symm}
On a conic special K\"{a}hler manifold $(\widetilde{M},\widetilde{g},\widetilde{I},\widetilde{\omega},\nabla,\xi)$, the tensor $\widetilde{\eta}$ is a section of $\sharp_2[\![S_{3,0}\widetilde{M}]\!]$.
\end{lemma}
In proving this lemma, one finds the following equality (see e.g.\ \cite[(3.3), p.\ 1743]{BauesCortes}), which we write for further reference.
\begin{equation}\label{eq:nabla I}
\nabla \widetilde{I}=[\widetilde{\eta},\widetilde{I}] =-2\widetilde{I}\widetilde{\eta}.
\end{equation}

Using the flatness of $\nabla$, we observe:
\begin{align}
0=\Omega^\nabla&=\widetilde{\Omega}^{LC}+\widetilde{d}^{LC}\widetilde{\eta} +\frac{1}{2}[\widetilde{\eta}\wedge\widetilde{\eta}],
\end{align}
where $\widetilde{\Omega}^{LC}$ and $\widetilde{d}^{LC}$ are respectively the curvature and exterior covariant derivative of the Levi-Civita connection on $\widetilde{M}$.

Arguing as in \cite[Proposition 1.34 (a), p.\ 39]{Freed1999} (see also \cite[Proposition 4 (iii), p.\ 1743]{BauesCortes}), one obtains
\begin{prop}\label{prop:SpezzamentoCurvatura}
For a K\"{a}hler manifold $(\widetilde{M},\widetilde{g},\widetilde{I},\widetilde{\omega})$ with a tensor $\widetilde{\eta}$ in $T^*M\otimes TM\otimes T^*M$ such that $\flat_2\widetilde{\eta}$ is a section of $[\![S_{3,0}\widetilde{M}]\!]$ and with a connection $\nabla$ with connection form $\omega^{\nabla}=\widetilde{\omega}^{LC}+\widetilde{\eta}$, then
\begin{align}
\Omega^{\nabla}=0&& \textrm{if and only if}&&\begin{cases}
\widetilde{\Omega}^{LC}+\frac{1}{2}[\widetilde{\eta}\wedge\widetilde{\eta}]=0\\
\widetilde{d}^{LC}\widetilde{\eta}=0
\end{cases}.
\end{align}
\end{prop}
\section{Conic and projective special K\"{a}hler metrics}
In this section we will consider the case of a projective special K\"{a}hler manifold $(\pi\colon\widetilde{M}\to M,\nabla)$ and we will give the explicit relation between the metric on $\widetilde{M}$ and the one on $M$ (see e.g.\ \cite[Section 1.1]{CortesCompProj}).

The mapping $\pi\colon\widetilde{M}\to M$ is a $\C^*$-principal bundle with infinitesimal principal action generated by $\xi$ and $\widetilde{I}\xi$.
We can always build the function $r=\sqrt{-\widetilde{g}(\xi,\xi)}\colon\widetilde{M}\to\R^+$ and define $S=r^{-1}(1)\subseteq \widetilde{M}$ with inclusion map $\iota_S\colon S\hookrightarrow \widetilde{M}$.
Now $r$ has no critical points, since
\begin{align}\label{eq:dr}
dr
&=\frac{d(r^2)}{2r}
=\frac{\widetilde{\nabla}^{LC}(r^2)}{2r}
=-\frac{\widetilde{\nabla}^{LC}(\widetilde{g}(\xi,\xi))}{2r}\\
&=-\frac{2\widetilde{g}(\widetilde{\nabla}^{LC} \xi,\xi)}{2r}
=-\frac{\widetilde{g}(\cdot,\xi)}{r}
=-\frac{1}{r}\xi^{\flat}
\end{align}
and $\widetilde{g}$ is non-degenerate.
It follows that $S$ is a submanifold of dimension $2n+1$ whose tangent bundle corresponds to $\ker(dr)\subset T\widetilde{M}$.
Notice that $dr(\widetilde{I}\xi)=-\frac{\widetilde{g}(\widetilde{I}\xi,\xi)}{r}=-\frac{\widetilde{\omega}(\xi,\xi)}{r}=0$, so $\widetilde{I}\xi$ is a vector field tangent to $S$ and it induces a principal $\U(1)$-action.
The induced metric on $S$ is $g_S=\iota_S^* \widetilde{g}$ and thus $\Lie{\widetilde{I}\xi}g_S=\iota_S^*\Lie{\widetilde{I}\xi}\widetilde{g}=0$.

The principal action of $\C^*$ on $\widetilde{M}$ induces by inclusion an $\R^+$-action, and in addition we have
\begin{lemma}
The map $r\colon\widetilde{M}\to\R^+$ is degree $1$ homogeneous with respect to the action of $\R^+\subseteq\C^*$ on $\widetilde{M}$, i.e.\ for all $s\in\R^+$ and $p\in\widetilde{M}$
\begin{equation}
r(ps)=r(p)s.
\end{equation}
\end{lemma}
As a consequence of this lemma, we can now define a retraction
\begin{align*}
p\colon\widetilde{M}&\longrightarrow S,\qquad u\longmapsto u\frac{1}{r(u)},
\end{align*}
which is well defined since $r(p(u))=r(u\frac{1}{r(u)})=\frac{r(u)}{r(u)}=1$.
Moreover, $p \iota_S=\id_S$ implies the surjectivity of $p$, which allows us to see $p\colon\widetilde{M}\to S$ as a principal $\R^+$-bundle and $\pi_S:=\pi\iota_S\colon S\to M$ as a principal $S^1$-bundle; the composition of the two gives $\pi$.

\begin{lemma}
If $(\pi\colon\widetilde{M}\to M,\nabla)$ is projective special K\"{a}hler, then $\widetilde{M}$ is diffeomorphic to $S\times \R^+$, and moreover
\begin{equation}
\widetilde{g}=r^2p^*g_S-dr^2.
\end{equation}
\end{lemma}
\begin{proof}
Let $a\colon S\times \R^+\to\widetilde{M}$ be the restriction of the principal right action $\widetilde{M}\times\R^+\to\widetilde{M}$ to $S\times \R^+$ and consider also $(p,r)\colon\widetilde{M}\to S\times\R^+$.
These maps are smooth and each an inverse to the other, in fact if $u\in\widetilde{M}$, $a(p,r)(u)=a(p(u),r(u))=u\frac{1}{r(u)}r(u)=u$ and for all $(q,s)\in S\times\R^+$, $(\pi_{S},r)a(q,s)=(p(qs),r(qs))=(q\frac{s}{r(qs)},r(q)s)=(q,s)$.

For the second statement consider the symmetric tensor
\begin{equation}
g'=\frac{1}{r^2}(\widetilde{g}+dr^2).
\end{equation}
We want to prove it is basic, that is horizontal and invariant with respect to the principal $\R^+$-action.

Since there is only one vertical direction, and since $g'$ is symmetric, it is enough to check whether $g'$ vanishes when evaluated on the fundamental vector field $\xi$ in one component.
Using \eqref{eq:dr} we obtain
\begin{align*}
g'(\xi,\cdot)
=\frac{1}{r}(\widetilde{g}(\xi,\cdot)+dr(\xi)dr)
=\frac{1}{r}(-rdr+r dr)
=0.
\end{align*}

And now for the $\R^+$-invariance:
\begin{align*}
\Lie{\xi}g'
&=-2\frac{\Lie{\xi}r}{r^3}(\widetilde{g}+dr^2)+\frac{1}{r^2}(\Lie{\xi}\widetilde{g}+2\Lie{\xi}(dr)dr)\\
&=-2\frac{dr(\xi)}{r^3}(\widetilde{g}+dr^2)+\frac{1}{r^2}(2\widetilde{g}+2(d\iota_\xi dr+\iota_\xi d^2r)dr)\\
&=-2\frac{r}{r^3}(\widetilde{g}+dr^2)+\frac{1}{r^2}(2\widetilde{g}+2dr^2)
=0.
\end{align*}

Therefore $g'$ is basic, which in turn implies it is of the form $p^*g''$ for some tensor $g''\in T_2S$, so that
\begin{equation}
\widetilde{g}=r^2p^*g''-dr^2.
\end{equation}

The proof is ended by the following observation:
\begin{equation*}
g_S
=\iota_S^* \widetilde{g}
=\iota_S^*\left(r^2p^*g''-dr^2\right)
=\iota_S^*p^*g''-\iota_S^*dr^2
=(p\iota_S)^*g''
=g''.
\qedhere
\end{equation*}
\end{proof}

The $\C^*$-bundle $\pi\colon\widetilde{M}\to M$ has a unique principal connection orthogonal to the fibres with respect to $\widetilde{g}$; the connection form can be written as
\begin{equation}\label{eq:C*connection}
\frac{dr}{r}+i\widetilde{\varphi}.
\end{equation}
Explicitly, we can describe $\widetilde{\varphi}$ using the metric:
\begin{equation}
\widetilde{\varphi}=\frac{\widetilde{g}(\widetilde{I}\xi,\cdot)}{\widetilde{g}(\widetilde{I}\xi,\widetilde{I}\xi)}=-\frac{1}{r^2}I\xi^{\flat}=-\frac{1}{r^2}\iota_{\xi}\widetilde{\omega}.
\end{equation}
If we restrict it to $S$, we obtain a connection form $\varphi=\iota_S^*\widetilde{\varphi}=-\iota_S^*(\iota_\xi\omega)$ corresponding to the $S^1$-action on $S$.

Notice that $p^*\varphi=\widetilde{\varphi}$, because the connection form \eqref{eq:C*connection} is right-invariant, so $\widetilde{\varphi}=p^*\varphi'$ for some $\varphi'$, and thus $\varphi=\iota_S^*\widetilde{\varphi}=\iota_S^*p^*\varphi'=(p\iota_S)^*\varphi'=\varphi'$.

The moment map for the action generated by $\widetilde{I}\xi$ is $\mu\colon\widetilde{M}\to \lie{u}(1)\cong\R$ such that $d\mu=\iota_{\widetilde{I}\xi}\omega=-\xi^{\flat}=rdr=d\left(\frac{r^2}{2}\right)$, so up to an additive constant, we can assume
\begin{equation}
\mu=\frac{r^2}{2}.
\end{equation}
Since $S=\mu^{-1}(\frac{1}{2})$ is a level set of the moment map and $M$ is the K\"{a}hler quotient, $\pi_{S}\colon S\to M$ is a pseudo-Riemannian submersion and thus we can write $g_S=\pi_{S}^*g-\varphi^2$.

\begin{prop}
A projective special K\"{a}hler manifold $(\pi\colon\widetilde{M}\to M,\nabla)$ satisfies
\begin{align*}
\widetilde{g}&=r^2\pi^*g-r^2\widetilde{\varphi}^2-dr^2,\\
\widetilde{\omega}&=r^2\pi^*\omega_M+r\widetilde{\varphi}\wedge dr.
\end{align*}
\end{prop}
\begin{proof}
From the previous arguments
\begin{align*}
\widetilde{g}
&=r^2 p^*g_S-dr^2
=r^2p^*(\pi_{S}^*g-\varphi^2)-dr^2\\
&=r^2(\pi_{S}p)^*g-r^2\widetilde{\varphi}^2-dr^2
=r^2\pi^*g-r^2\widetilde{\varphi}^2-dr^2.
\end{align*}
For the K\"{a}hler form it is enough to notice that $\pi$ is holomorphic, $M$ being a K\"{a}hler quotient, and that
\begin{equation*}
(r\widetilde{\varphi})\circ\widetilde{I}
=-\frac{1}{r}\widetilde{I}\xi^{\flat}\widetilde{I}
=-\frac{1}{r}\xi^{\flat}
=dr.
\qedhere
\end{equation*}
\end{proof}

For future reference we give the following
\begin{rmk}\label{rmk:curvatureSbundle}
The curvature of $\varphi$ is computed using Lemma \ref{lemma:azioneInfinitesimaXieIxi}:
\begin{equation}
d\varphi=-d\iota_S^*\iota_{\xi}\widetilde{\omega}=\iota_S^*(-\Lie{\xi}\widetilde{\omega}+\iota_{\xi}d\widetilde{\omega})=-2\iota_S^*\widetilde{\omega}=-2\pi_S^*\omega_M.
\end{equation}
in fact, the restriction to $S$ of $\widetilde{\omega}$ maps fixes $r=1$ and thus kills $dr$.

It will be useful to compute also
\begin{equation}
d\widetilde{\varphi}=-2\pi^*\omega_M.
\end{equation}
\end{rmk}
\section{Lifting the coframe}\label{sec:coframe}
The purpose of this section is to lift a generic unitary coframe on a projective special K\"{a}hler manifold to one on the corresponding conic special K\"{a}hler.
This will enable us to give a more explicit formulation of the Levi-Civita connection and associated curvature tensor on the conic special K\"{a}hler manifold.

In our convention, on a K\"{a}hler manifold $(M,g,I,\omega)$, the Hermitian form is $h=g+i\omega$.
Given a projective special K\"{a}hler manifold $(\pi\colon\widetilde{M}\to M,\nabla)$ and an open subset $U\subseteq M$, consider a unitary coframe $\theta=(\theta^1,\dots,\theta^n)\in \Omega^1(U,\C^n)$ on $M$, then we can build a coframe $\widetilde{\theta}\in \Omega^1(\pi^{-1}(U),\C^{n+1})$ on $\widetilde{M}$ as follows:
\begin{equation}\label{eq:coframe Mtilde}
\widetilde{\theta}^k=\begin{cases}
r\pi^*\theta^k&\mbox{if }k\le n\\
dr+i r\widetilde{\varphi}&\mbox{if }k=n+1
\end{cases}.
\end{equation}
This coframe is compatible with the $\U(n,1)$-structure because it takes complex values and
\begin{equation}
\sum_{k=1}^n\overline{\widetilde{\theta}^k}\widetilde{\theta}^k-\overline{\widetilde{\theta}^{n+1}}\widetilde{\theta}^{n+1}
=r^2\pi^*\left(\sum_{k=1}^n\overline{\theta^k}\theta^k\right)-dr^2-r^2\widetilde{\varphi}^2
=\widetilde{g}.
\end{equation}

We will denote the dual frame to a given coframe by the same symbol, but with lower indices.

\begin{rmk}
Let $T=\C^{n+1}$ be the standard real representation of $\U(n,1)$, and let $T\otimes_{\R}\C \cong T^{1,0}\oplus T^{0,1}$ be the holomorphic, anti-holomorphic split.
Given a connection on a K\"{a}hler manifold, it can be represented by a connection form $\omega$ with values in $\lie{u}(n,1)$ whose complexification is $\lie{gl}(n+1,\C)\cong T^{1,0}\otimes T_{1,0}\oplus T^{0,1}\otimes T_{1,0}$, so we obtain projections in each component, respectively $\omega^{1,0}_{1,0}$ and $\omega^{0,1}_{0,1}$ such that $\omega=\omega^{1,0}_{1,0}+\omega^{0,1}_{0,1}$.
Notice that $\omega^{0,1}_{0,1}=\overline{\omega^{1,0}_{1,0}}$ because $\omega$ comes from a real representation and to give the first component is equivalent to give the whole form.
Notice also that $([\![T]\!],I)$, as complex representation, is isomorphic to $T^{1,0}$ and the component $A^{1,0}_{1,0}$ of an endomorphism $A$ gives the corresponding endomorphism of $T^{1,0}$.
We will often present connection forms by giving only the $T^{1,0}_{1,0}$ component.

We will call $\RE$ the projection from the complex tensor algebra to the real representation, defined so that $\RE(\alpha)=\alpha+\overline{\alpha}$ where the conjugate is the real structure.
\end{rmk}
\begin{prop}\label{prop:conicLC}
Let $(\pi\colon\widetilde{M}\to M,\nabla)$ be a projective special K\"{a}hler manifold, let $(U,\theta)$ be a local unitary coframe on $M$ lifted as in \eqref{eq:coframe Mtilde} to a coframe $\widetilde{\theta}$ adapted to the $\U(n,1)$-structure on $\widetilde{M}$.
With respect to $\widetilde{\theta}$, the Levi-Civita connection form on $\widetilde{M}$ is represented by
\begin{equation}
\widetilde{\omega}^{LC}
=\begin{pmatrix}
\pi^*\omega^{LC}& \rvline&0\\
\hline
0& \rvline&0
\end{pmatrix}+\frac{1}{r}\begin{pmatrix}
i\Im\left(\widetilde{\theta}^{n+1}\right)&&0& \rvline&\widetilde{\theta}^1\\
&\ddots&& \rvline&\vdots\\
0&&i\Im\left(\widetilde{\theta}^{n+1}\right)& \rvline&\widetilde{\theta}^n\\
\hline
\overline{\widetilde{\theta}^1}&\cdots&\overline{\widetilde{\theta}^n}& \rvline&i\Im\left(\widetilde{\theta}^{n+1}\right)
\end{pmatrix},
\end{equation}
that is
\begin{equation}\label{eq:LeviCivita}
\widetilde{\omega}^{LC}
=\begin{pmatrix}
\pi^*\omega^{LC}+i\widetilde{\varphi}\otimes I_n& \rvline&\pi^*\theta\\
\hline
\pi^*\theta^{\star}& \rvline&i\widetilde{\varphi}
\end{pmatrix}
\end{equation}
and its curvature form is
\begin{equation}
\widetilde{\Omega}^{LC}=\begin{pmatrix}
\pi^*(\Omega^{LC}+\theta\wedge\theta^*-2i\omega_M \otimes\id)& \rvline&0\\
\hline
0& \rvline&0
\end{pmatrix}.
\end{equation}
\end{prop}
\begin{proof}
The connection form \eqref{eq:LeviCivita} is metric if and only if the matrix is anti-Hermitian with respect to $\widetilde{g}$ and since $\omega^{LC}$ is anti-Hermitian with respect to $g$, we get
\begin{align*}
(\widetilde{\omega}^{LC})^{\star}
=\begin{pmatrix}
\pi^*(\omega^{LC})^{\star}-i\widetilde{\varphi}\otimes I_n& \rvline&-\pi^*\theta\\
\hline
-\pi^*\theta^{\star}& \rvline&-i\widetilde{\varphi}
\end{pmatrix}
=-\widetilde{\omega}^{LC}.
\end{align*}

The torsion form of this connection is $\widetilde{\Theta}^{LC}=d\widetilde{\theta}+\widetilde{\omega}^{LC}\wedge\widetilde{\theta}$, so for $1\le k\le n$
\begin{align*}
\left(\widetilde{\Theta}^{LC}\right)^k
&=d\widetilde{\theta}^k+\sum_{j=1}^{n}\left(\widetilde{\omega}^{LC}\right)^k_j\wedge
\widetilde{\theta}^j+\left(\widetilde{\omega}^{LC}\right)^k_{n+1}\wedge\widetilde{\theta}^{n+1}\\
&=d\left(r\pi^*\theta^k\right)+\sum_{j=1}^{n}\left(\pi^*(\omega^{LC})^k_j+i\widetilde{\varphi}\delta^k_j\right)\wedge
\left(r\pi^*\theta^j\right)+\pi^*\theta^k\wedge\widetilde{\theta}^{n+1}\\
&=r\pi^*(\Theta^{LC})^k+(dr+ir\widetilde{\varphi})\wedge \pi^*\theta^k+\pi^*\theta^k\wedge\widetilde{\theta}^{n+1}\\
&=0+\widetilde{\theta}^{n+1}\wedge \pi^*\theta^k+\pi^*\theta^k\wedge\widetilde{\theta}^{n+1}
=0.
\end{align*}
In the last component
\begin{align*}
(\widetilde{\Theta}^{LC})^{n+1}
&=d\widetilde{\theta}^{n+1}+\sum_{j=1}^n \pi^*\overline{\theta^j}\wedge r\pi^*\theta^j+i\widetilde{\varphi}\wedge \widetilde{\theta}^{n+1}\\
&=d(dr+ir\widetilde{\varphi})+r\pi^*\left(\sum_{j=1}^n\overline{\theta^j}\wedge\theta^j\right)+i\widetilde{\varphi}\wedge \widetilde{\theta}^{n+1}\\
&=idr\wedge\widetilde{\varphi}+ir(d\widetilde{\varphi}+2\pi^*\omega_M)+i\widetilde{\varphi}\wedge dr
=0.
\end{align*}
$\widetilde{\omega}^{LC}$ is metric and torsion-free, therefore by uniqueness it must be the Levi-Civita connection.

Now let us compute its curvature form $\widetilde{\Omega}^{LC}=d\widetilde{\omega}^{LC}+\widetilde{\omega}^{LC}\wedge\widetilde{\omega}^{LC}$.
For $1\le k,h\le n$ we have
\begin{align*}
\left(\widetilde{\Omega}^{LC}\right)^h_k
&=d(\widetilde{\omega}^{LC})^h_k+(\widetilde{\omega}^{LC})^h_j\wedge(\widetilde{\omega}^{LC})^j_k\\
&=d\pi^*(\omega^{LC})^h_k+id\widetilde{\varphi}\delta^h_k+\sum_{j=1}^n(\pi^*(\omega^{LC})^h_j+i\widetilde{\varphi}\delta^h_j)\wedge (\pi^*(\omega^{LC})^j_k+i\widetilde{\varphi}\delta^j_k)\\
&\quad +\pi^*\theta^h\wedge\pi^*\overline{\theta^k}\\
&=\pi^*d(\omega^{LC})^h_k-2i\pi^*\omega_M\delta^h_k+\pi^*((\omega^{LC})^h_j\wedge (\omega^{LC})^j_k\\
&\quad +i\widetilde{\varphi}\wedge \pi^*(\omega^{LC})^h_k+\pi^*(\omega^{LC})^h_k\wedge i\widetilde{\varphi}-\widetilde{\varphi}\wedge \widetilde{\varphi}\delta^h_k+\pi^*\theta^h\wedge\pi^*\overline{\theta^k}\\
&=\pi^*(\Omega^{LC})^h_k-2i\pi^*\omega_M\delta^h_k+\pi^*(\theta^h\wedge\overline{\theta^k})
\end{align*}
and
\begin{align*}
\left(\widetilde{\Omega}^{LC}\right)^h_{n+1}
&=d\pi^*\theta^h+\sum_{j=1}^n(\pi^*(\omega^{LC})^h_j+i\widetilde{\varphi}\delta^h_j)\wedge \pi^*\theta^j+\pi^*\theta^h\wedge i\widetilde{\varphi}=\pi^*\left(\Theta^{LC}\right)^h\\*
&=0.
\end{align*}
Since the curvature form must also be anti-Hermitian, we also get
\begin{align*}
\left(\widetilde{\Omega}^{LC}\right)^{n+1}_k
&=-\left(\left(\widetilde{\Omega}^{LC}\right)^\star\right)^{n+1}_k
=\overline{\left(\widetilde{\Omega}^{LC}\right)^k_{n+1}}
=0.
\end{align*}
Finally,
\begin{equation*}
\left(\widetilde{\Omega}^{LC}\right)^{n+1}_{n+1}
=id\widetilde{\varphi}+\sum_{j=1}^n\pi^*\overline{\theta^j}\wedge\pi^*\theta^j-\widetilde{\varphi}\wedge \widetilde{\varphi}
=id\widetilde{\varphi}+2i\pi^*\omega_M
=0.\qedhere
\end{equation*}
\end{proof}
\begin{rmk}\label{rmk:curvatureProj}
The tensor $\theta\wedge\theta^{\star}-2i\omega_M\otimes id$, or explicitly
\begin{equation}
\Omega_{\Pj^n_{\C}}:=\RE\left((\theta^k\wedge\overline{\theta^h})\otimes \theta_k\otimes \theta^{h}-(\overline{\theta^k}\wedge\theta^k)\otimes \theta_h\otimes \theta^h\right)
\end{equation}
is a curvature tensor of the complex projective space of dimension $n$; in fact, $\Omega_{\Pj_{\C}^n}$ is the curvature with respect to the Fubini-Study metric (see for example \cite[II, p.\ 277]{KN}).
In order to verify that $\Omega_{\Pj^n_{\C}}$ is exactly the curvature of the Fubini-Study rather than a multiple, we compute the Ricci tensor:
\begin{equation}\label{eq:RicciProj}
\Ric_{\Pj^n_{\C}}
=\RE\left(n\theta^h\otimes \overline{\theta^h}+\delta_{h,k}\theta^h\otimes\overline{\theta^k}\right)
=\RE\left((n+1)h\right)
=2(n+1)g.
\end{equation}
Then,
\begin{equation}\label{eq:ScalProj}
\scal_{\Pj_{\C}^n}=2(n+1).
\end{equation}
Thus $\Omega_{\Pj^n_{\C}}$ corresponds exactly to the curvature of $\Pj_{\C}^n$ with the Fubini-Study metric.
\end{rmk}

Now, whenever we have a smooth map $f\colon M\to N$ between Riemannian manifolds, we can extend the pull-back $f^*\colon T_\bullet N\to T_\bullet M$ on the covariant tensor algebra to the whole tensor algebra, using the musical isomorphisms in each contravariant component.
Explicitly, for $X$ vector field on $N$, we define $f^*X:=\sharp f^*\flat X=(f^*X^{\flat})_\sharp$.
Notice that this extension of the pull-back is still functorial, since if $f\colon M\to N$, $g\colon N\to L$ are smooth maps, then $f^*g^*X=\sharp f^*\flat \sharp g^*\flat X=\sharp f^* g^*\flat X=\sharp (gf)^*\flat X=(gf)^* X$.

Since $\widetilde{M}$ and $M$ are Riemannian manifolds, we have $\pi^*\colon T^{\bullet}_{\bullet}M\to T^{\bullet}_{\bullet}\widetilde{M}$, and in particular, for $1\le k\le n$ we have
\begin{align*}
\pi^*\theta_k
=(\pi^*\theta_k^{\flat})_\sharp
=\frac{1}{2}(\pi^*\overline{\theta^k})_\sharp
=\frac{1}{2r}(\overline{\widetilde{\theta}^k})_\sharp
=\frac{1}{r}\widetilde{\theta}_k.
\end{align*}

\begin{rmk}\label{rmk:Levi-Civita_CSK}
In this notation,
\begin{equation}
\widetilde{\Omega}^{LC}=r^2\pi^*(\Omega^{LC}+\Omega_{\Pj_{\C}^n}).
\end{equation}
\end{rmk}

\section{Deviance}
In this section we will continue the analysis of the tensor $\widetilde{\eta}$ started in section \ref{sec:differenceTensor}.
The aim is to reduce it to a locally defined tensor on $M$ that we call deviance.
We will then use it to give an explicit local description of the Ricci tensor and the scalar curvature.

\begin{lemma}\label{lemma:etaHorizontal}
On a projective special K\"{a}hler manifold $(\pi\colon\widetilde{M}\to M,\nabla)$, if $\widetilde{\eta}_X Y=\nabla_X Y-\widetilde{\nabla}^{LC}_X Y$, then $\flat_2\widetilde{\eta}$ is horizontal with respect to $\pi$.

In other words, $\flat_2(\widetilde{\eta})$ is a section of $\pi^*[\![S_{3,0}M]\!]\subset [\![S_{3,0}\widetilde{M}]\!]$.
Explicitly, $\widetilde{\eta}_v$, $\widetilde{\eta} v$ and $\widetilde{g}(\widetilde{\eta}, v)$ vanish for all $v\in\langle \xi,I\xi\rangle$.
\end{lemma}
\begin{proof}
First notice that $\widetilde{\eta} (\xi)=\nabla \xi- \widetilde{\nabla}^{LC} \xi=0$, so by symmetry $\widetilde{\eta}_{\xi}=0$ and $g(\eta,\xi)=0$, so $\flat_2(\widetilde{\eta})$ in each component vanishes when evaluated at $\xi$.
From this fact and \eqref{eq:nabla I}, we also deduce $\widetilde{\eta} (\widetilde{I}\xi)=\widetilde{I}\widetilde{\eta} (\xi)+[\widetilde{\eta},\widetilde{I}] \xi=0-2\widetilde{I} \widetilde{\eta} (\xi)=0$.
By symmetry, we conclude that $\flat_2\widetilde{\eta}$ vanishes in every component on $I\xi$.
Linearity then completes the proof.
\end{proof}

\begin{lemma}\label{lemma:LieDerivativesEta}
Let $(\widetilde{M},\widetilde{g},\widetilde{I},\widetilde{\omega},\nabla,\xi)$ be a conic special K\"{a}hler manifold and $\widetilde{\eta}$ be as above, then
\begin{enumerate}
\item $\Lie{\xi}\widetilde{\eta}=0$;
\item $\Lie{\widetilde{I}\xi}\widetilde{\eta}=-2\widetilde{I}\widetilde{\eta}$.
\end{enumerate}
\end{lemma}
\begin{proof}
The proof relies on a generic formula satisfied by a torsion-free connection $D$ (see e.g.\ \cite[equation (3.1), p.\ 1336]{Swann2015}), that is:
\begin{align*}
\Lie{A}(D_X Y)&-D_{\Lie{A}X} Y-D_X \Lie{A}Y=\Omega^D(A,X)Y-D_{D_{X}Y}A+D_{X}D_{Y} A.
\end{align*}
\begin{enumerate}
\item We check the formula on vector fields $X,Y\in\field{\widetilde{M}}$
\begin{align*}
(\Lie{\xi}\widetilde{\eta})_{X} Y
&=\Lie{\xi}(\widetilde{\eta}_{X} Y)-\widetilde{\eta}_{\Lie{\xi} X}Y-\widetilde{\eta}_{X}\Lie{\xi} Y\\
&=\Lie{\xi}\nabla_X Y-\Lie{\xi}\widetilde{\nabla}^{LC}_X Y-\nabla_{\Lie{\xi}X} Y+\widetilde{\nabla}^{LC}_{\Lie{\xi}X} Y\\
&\quad -\nabla_X \Lie{\xi}Y+\widetilde{\nabla}^{LC}_X \Lie{\xi}Y\\
&=\Omega^{\nabla}(\xi,X)Y-\nabla_{\nabla_{X}Y}\xi+\nabla_{X}\nabla_{Y} \xi-\widetilde{\Omega}^{LC}(\xi,X)Y\\
&\quad +\widetilde{\nabla}^{LC}_{\widetilde{\nabla}^{LC}_{X}Y}\xi-\widetilde{\nabla}^{LC}_{X}\widetilde{\nabla}^{LC}_{Y} \xi\\
&= -\nabla_{X}Y+\nabla_{X}Y-\widetilde{\Omega}^{LC}(\xi,X)Y+\widetilde{\nabla}^{LC}_{X}Y-\widetilde{\nabla}^{LC}_{X}Y\\
&=-\widetilde{\Omega}^{LC}(\xi,X)Y.
\end{align*}
Lowering the contravariant index of the curvature form, for $Z\in\field{\widetilde{M}}$, thanks to the symmetries of the Riemannian tensor we obtain
\begin{align*}
\widetilde{g}\left(\widetilde{\Omega}^{LC}(\xi,X)Y,Z\right)
&=\widetilde{g}\left(\widetilde{\Omega}^{LC}(Y,Z)\xi,X\right)\\
&=\widetilde{g}\left(\widetilde{\nabla}^{LC}_Y \widetilde{\nabla}^{LC}_Z\xi-\widetilde{\nabla}^{LC}_Z \widetilde{\nabla}^{LC}_Y\xi-\widetilde{\nabla}^{LC}_{[Y,Z]}\xi,X\right)\\
&=\widetilde{g}\left(\widetilde{\nabla}^{LC}_Y Z-\widetilde{\nabla}^{LC}_Z Y-[Y,Z],X\right)\\
&=\widetilde{g}\left(\Theta^{LC}(Y,Z),X\right)
=0,
\end{align*}
proving that $\widetilde{\Omega}^{LC}(\xi,X)Y=0$, which implies the statement.
\item As before,
\begin{align*}
(\Lie{\widetilde{I}\xi}\widetilde{\eta})_{X} Y
&=\Omega^{\nabla}(\widetilde{I}\xi,X)Y-\nabla_{\nabla_{X}Y}(\widetilde{I}\xi)+\nabla_{X}\nabla_{Y} (\widetilde{I}\xi)-\widetilde{\Omega}^{LC}(\widetilde{I}\xi,X)Y\\
&\quad +\widetilde{\nabla}^{LC}_{\widetilde{\nabla}^{LC}_{X}Y}(\widetilde{I}\xi)-\widetilde{\nabla}^{LC}_{X}\widetilde{\nabla}^{LC}_{Y} (\widetilde{I}\xi)\\
&=-\widetilde{I}\nabla_{X}Y+\nabla_{X}(\widetilde{I}Y)-\widetilde{\Omega}^{LC}(\widetilde{I}\xi,X)Y+\widetilde{I}\widetilde{\nabla}^{LC}_{X}Y-\widetilde{\nabla}^{LC}_{X}(\widetilde{I}Y)\\
&=(\nabla \widetilde{I})(X,Y)-\widetilde{\Omega}^{LC}(\widetilde{I}\xi,X)Y.
\end{align*}
Proceeding as in the previous point
\begin{align*}
\widetilde{g}&\left(\widetilde{\Omega}^{LC}(\widetilde{I}\xi,X)Y,Z\right)
=\widetilde{g}\left(\widetilde{\Omega}^{LC}(Y,Z)(\widetilde{I}\xi),X\right)\\
&=\widetilde{g}\left(\widetilde{\nabla}^{LC}_Y \widetilde{\nabla}^{LC}_Z(\widetilde{I}\xi)-\widetilde{\nabla}^{LC}_Z \widetilde{\nabla}^{LC}_Y(\widetilde{I}\xi)-\widetilde{\nabla}^{LC}_{[Y,Z]}(\widetilde{I}\xi),X\right)\\
&=\widetilde{g}\left(\widetilde{I}\widetilde{\Omega}^{LC}(Y,Z)\xi,X\right)
=-\widetilde{g}\left(\widetilde{\Omega}^{LC}(Y,Z)\xi,IX\right)\\
&=-\widetilde{g}\left(\widetilde{\Omega}^{LC}(\xi,\widetilde{I}X)Y,Z\right).
\end{align*}
This quantity is zero as shown in the previous point, so it follows that $\Lie{\widetilde{I}\xi}\widetilde{\eta}=\nabla \widetilde{I}$, so \eqref{eq:nabla I} ends the proof.\qedhere
\end{enumerate}
\end{proof}

We can now use a coframe $\widetilde{\theta}$ as in section \ref{sec:coframe} in order to progress in the study of $\widetilde{\eta}$.
We then write
\begin{equation}
\widetilde{\eta}=\RE(\widetilde{\eta}^j_{k,h}\widetilde{\theta}^k\otimes \overline{\widetilde{\theta}_j}\otimes \widetilde{\theta}^h).
\end{equation}
Since every operator we use is $\C$-linear, we can study only the component in $T_{1,0} \otimes T^{0,1}\otimes T_{1,0}$, that is $\widetilde{\eta}^j_{k,h}\widetilde{\theta}^k\otimes \overline{\widetilde{\theta}_j}\otimes \widetilde{\theta}^h$.
Because of Lemma \ref{lemma:etaHorizontal}, the coefficients $\widetilde{\eta}^{j}_{k,h}$ vanish if any one of the indices is $n+1$; moreover, $\widetilde{\eta}^j_{k,h}$ is completely symmetric in its indices.
The last statement follows from the fact that $\flat_2\widetilde{\eta}$ is a tensor in $\pi^*S_{3,0}M$, and such tensors are expressed using only $\pi^*\theta^k$ for $1\le k\le n$, where the metric is positive definite, and thus $\flat_2$ does not change the signs of the coefficients of $\widetilde{\eta}$.

We are now ready to reduce $\widetilde{\eta}$ to an object defined locally on the base space.
\begin{prop}\label{prop:ThetatildeWrtTheta}
Given a projective special K\"{a}hler $(\pi\colon\widetilde{M}\to M,\nabla)$ and a section $s\colon U\to S\subseteq\widetilde{M}$ inducing a trivialisation $(\pi|_{\pi^{-1}(U)},z)\colon \pi^{-1}(U)\to U\times\C^*$, there exists a tensor $\eta$ in $T_{1,0}U\otimes T^{0,1}U\otimes T_{1,0}U$ such that $\flat_2 \eta$ is a tensor in $S_{3,0}U$ and
\begin{equation}
\widetilde{\eta}=\RE(z^2\pi^*\eta)=r^2\cos(2\vartheta)2\Re\pi^*\eta+r^2\sin(2\vartheta)2\Im\pi^*\eta
\end{equation}
where $z=r e^{i\vartheta}$.
\end{prop}
\begin{proof}
For every point $p\in M$ we can find a local unitary coframe $\theta$ defined on an open set containing $p$, and the corresponding coframe $\widetilde{\theta}$ on $\widetilde{M}$ as in \eqref{eq:coframe Mtilde}.

For the coming arguments we first compute the following Lie derivatives
\begin{align*}
\Lie{\xi}\widetilde{\theta^k}
&=d\iota_{\xi}(r\pi^*\theta^k)+\iota_{\xi}d(r\pi^*\theta^k)
=0+\iota_{\xi}(dr\wedge\pi^*\theta^k)+r\iota_{\xi}d\pi^*\theta^k\\*
&=dr(\xi) \pi^*\theta^k+r\iota_{\xi}\pi^*d\theta^k
=r \pi^*\theta^k+0
=\widetilde{\theta}^k;\\
\\
\Lie{\xi}\widetilde{\theta}_k
&=\widetilde{g}(\Lie{\xi}\widetilde{\theta}_k,\cdot)_{\sharp}
=\Lie{\xi}\left(\widetilde{g}(\widetilde{\theta}_k,\cdot)\right)_{\sharp}-\left(\Lie{\xi}\widetilde{g}(\widetilde{\theta}_k,\cdot)\right)_{\sharp}\\
&=\frac{1}{2}\left(\Lie{\xi}\overline{\widetilde{\theta}^k}\right)_{\sharp}-2\widetilde{g}(\widetilde{\theta}_k,\cdot)_{\sharp}
=\frac{1}{2}\overline{\widetilde{\theta}^k}_{\sharp}-2\widetilde{\theta}_k
=-\widetilde{\theta}_k;\\
\\
\Lie{\widetilde{I}\xi}\widetilde{\theta^k}
&=d\iota_{\widetilde{I}\xi}\widetilde{\theta}^k+\iota_{\widetilde{I}\xi}d\widetilde{\theta}^k
=d\iota_{\widetilde{I}\xi}(r\pi^*\theta^k)+\iota_{\widetilde{I}\xi}d(r\pi^*\theta^k)\\
&=0+r\iota_{\widetilde{I}\xi}d\pi^*\theta^k
=r\iota_{\widetilde{I}\xi}\pi^*d\theta^k
=0;\\
\\
\Lie{\widetilde{I}\xi}\widetilde{\theta}_k
&=\widetilde{g}(\Lie{\widetilde{I}\xi}\widetilde{\theta}_k,\cdot)_{\sharp}
=\Lie{\widetilde{I}\xi}\left(\widetilde{g}(\widetilde{\theta}_k,\cdot)\right)_{\sharp}
=\frac{1}{2}\left(\Lie{\xi}\overline{\widetilde{\theta}^k}\right)_{\sharp}
=0.
\end{align*}

Lemma \ref{lemma:LieDerivativesEta} implies
\begin{align*}
0=\Lie{\xi}\widetilde{\eta}
&=\Lie{\xi}\RE\left(\widetilde{\eta}^{j}_{k,h}\widetilde{\theta}^k\otimes\overline{\widetilde{\theta}_j}\otimes \widetilde{\theta}^h\right)\\
&=\RE\left(\Lie{\xi}\widetilde{\eta}^{j}_{k,h}\widetilde{\theta}^k\otimes\overline{\widetilde{\theta}_j}\otimes \widetilde{\theta}^h
+\widetilde{\eta}^{j}_{k,h}\Lie{\xi}\widetilde{\theta}^k\otimes\overline{\widetilde{\theta}_j}\otimes \widetilde{\theta}^h\right.\\
&\quad +\left.\widetilde{\eta}^{j}_{k,h}\widetilde{\theta}^k\otimes\Lie{\xi}\overline{\widetilde{\theta}_j}\otimes \widetilde{\theta}^h
+\widetilde{\eta}^{j}_{k,h}\widetilde{\theta}^k\otimes\overline{\widetilde{\theta}_j}\otimes \Lie{\xi}\widetilde{\theta}^h\right)\\
&=\RE\left(\Lie{\xi}\widetilde{\eta}^{j}_{k,h}\widetilde{\theta}^k\otimes\overline{\widetilde{\theta}_j}\otimes \widetilde{\theta}^h
+\widetilde{\eta}^{j}_{k,h}\widetilde{\theta}^k\otimes\overline{\widetilde{\theta}_j}\otimes \widetilde{\theta}^h\right)\\
&=\RE\left(\left(\Lie{\xi}\widetilde{\eta}^{j}_{k,h}+\widetilde{\eta}^{j}_{k,h}\right)\widetilde{\theta}^k\otimes\overline{\widetilde{\theta}_j}\otimes \widetilde{\theta}^h\right).
\end{align*}
and
\begin{align*}
0=\Lie{\widetilde{I}\xi}\widetilde{\eta}+2\widetilde{I}\widetilde{\eta}
&=\Lie{\widetilde{I}\xi}\RE\left(\widetilde{\eta}^{j}_{k,h}\widetilde{\theta}^k\otimes\overline{\widetilde{\theta}_j}\otimes \widetilde{\theta}^h\right)+\RE\left(2\widetilde{\eta}^{j}_{k,h}\widetilde{\theta}^k\otimes \widetilde{I}\left(\overline{\widetilde{\theta}_j}\right)\otimes \widetilde{\theta}^h\right)\\
&=\RE\left(\Lie{\xi}\widetilde{\eta}^{j}_{k,h}\widetilde{\theta}^k\otimes\overline{\widetilde{\theta}_j}\otimes \widetilde{\theta}^h
-2i\widetilde{\eta}^{j}_{k,h}\widetilde{\theta}^k\otimes\overline{\widetilde{\theta}_j}\otimes \widetilde{\theta}^h\right)\\
&=\RE\left(\left(\Lie{I\xi}\widetilde{\eta}^{j}_{k,h}-2i\widetilde{\eta}^{j}_{k,h}\right)\widetilde{\theta}^k\otimes\overline{\widetilde{\theta}_j}\otimes \widetilde{\theta}^h\right).
\end{align*}
Independent components must vanish, so we obtain a family of differential equations for $1\le j,k,h\le n$
\begin{equation}\label{eq:EqDiffEtatilde}
\begin{cases}
\Lie{\xi}\widetilde{\eta}^{j}_{k,h}=-\widetilde{\eta}^{j}_{k,h}\\
\Lie{\widetilde{I}\xi}\widetilde{\eta}^{j}_{k,h}=2i\widetilde{\eta}^{j}_{k,h}
\end{cases}.
\end{equation}

We define $\eta$, as the component in $T_{1,0}M\otimes T^{0,1}M\otimes T_{1,0}M$ of $s^*\widetilde{\eta}$, so that $\RE(\eta)=s^*\widetilde{\eta}$.

Notice that since $\pi s=\id_M$, the pullbacks satisfy $s^*\pi^*=\id_{T^{\bullet}_{\bullet}M}$, so
\begin{align*}
s^*\widetilde{\eta}
&=s^*\RE(\widetilde{\eta}^j_{k,h}\widetilde{\theta}^k\otimes\overline{\widetilde{\theta}_j}\otimes\widetilde{\theta}^h)
=\RE(s^*(r^3\widetilde{\eta}^j_{k,h}\pi^*\theta^k\otimes\pi^*\overline{\theta_j}\otimes\pi^*\theta^h))\\
&=\RE((r\circ s)^3(\widetilde{\eta}^j_{k,h}\circ s) s^*\pi^*\theta^k\otimes s^*\pi^*\overline{\theta_j}\otimes s^*\pi^*\theta^h)\\
&=\RE((\widetilde{\eta}^j_{k,h}\circ s) \theta^k\otimes \overline{\theta_j}\otimes \theta^h).
\end{align*}
Thus $\eta=s^*\widetilde{\eta}^j_{k,h}\theta^k\otimes \overline{\theta_j}\otimes \theta^h$ and we define $\eta^j_{k,h}:=s^*\widetilde{\eta}^j_{k,h}$.

Now we will use \eqref{eq:EqDiffEtatilde} to find $\widetilde{\eta}^j_{k,h}$ at a point of $\pi^*U$.
We define the function $f\colon \R\to\C$ such that $f(t):=\widetilde{\eta}^j_{k,h}(s(u)e^t)$ for $u\in U$ and compute its derivative at $t_0\in\R$.
\begin{align*}
\frac{d}{dt}f|_{t_0}
&=\frac{d}{dt}\widetilde{\eta}^j_{k,h}(s(u)e^t)|_{t=t_0}
=\frac{d}{dt}\widetilde{\eta}^j_{k,h}(s(u)e^{t_0+t})|_{t=0}
=\frac{d}{dt}\widetilde{\eta}^j_{k,h}(\phi_{\xi}^{t}(s(u)e^{t_0}))|_{t=0}\\
&=(\Lie{\xi}\widetilde{\eta}^j_{k,h})(s(u)e^{t_0})
=-\widetilde{\eta}^j_{k,h}(s(u)e^{t_0})
=-f(t_0).
\end{align*}
Moreover, $f(0)=\widetilde{\eta}^j_{k,h}(s(u))=\eta^j_{k,h}(u)$, so $f$ satisfies the following initial value problem
\begin{equation}
\begin{cases}
f'=-f\\
f(0)=\eta^j_{k,h}(u)
\end{cases}
\end{equation}
which has a unique solution, that is $f(t)=\eta^j_{k,h}(u)e^{-t}$.
This means that $\widetilde{\eta}^j_{k,h}(s(u)e^t)=\eta^j_{k,h}(u)e^{-t}$ or equivalently, for all $\rho\in\R^+$ we have $\widetilde{\eta}^j_{k,h}(s(u)\rho)=\frac{1}{\rho}\eta^j_{k,h}(u)=(\frac{1}{r}\pi^*\eta^j_{k,h})(s(u)\rho)$.

Similarly, consider the function $f\colon \R\to\C$ such that $f(t):=\widetilde{\eta}^j_{k,h}(s(u)\rho e^{it})$ and compute its derivative at $t_0\in\R$.
\begin{align*}
\frac{d}{dt}f|_{t_0}
&=\frac{d}{dt}\widetilde{\eta}^j_{k,h}(s(u)\rho e^{it})|_{t=t_0}
=\frac{d}{dt}\widetilde{\eta}^j_{k,h}(s(u)\rho e^{it_0+it})|_{t=0}\\
&=\frac{d}{dt}\widetilde{\eta}^j_{k,h}(\phi_{I\xi}^{t}(s(u)\rho e^{it_0}))|_{t=0}
=(\Lie{I\xi}\widetilde{\eta}^j_{k,h})(s(u)\rho e^{t_0})\\
&=2i\widetilde{\eta}^j_{k,h}(s(u)\rho e^{t_0})
=2if(t_0).
\end{align*}
And this time, $f(0)=\widetilde{\eta}^j_{k,h}(s(u)\rho)=\frac{1}{\rho}\eta^j_{k,h}(u)$, so that for $f$
\begin{equation}
\begin{cases}
f'=2if\\
f(0)=\frac{1}{\rho}\eta^j_{k,h}(u)
\end{cases}.
\end{equation}
Its unique solution is $f(t)=\eta^j_{k,h}(u)\frac{e^{2it}}{\rho}$, which implies 
\begin{equation}
\widetilde{\eta}^j_{k,h}(s(u)\rho e^{it})=\eta^j_{k,h}(u)\frac{e^{2it}}{\rho}=\left(\frac{\pi^*\eta^j_{k,h}}{r^3}\right)(s(u)\rho e^{it})\rho^2 e^{2it}.
\end{equation}

Let now $z\colon \pi^{-1}(U)\to \C^*$ be as in the statement, then in particular for all $w\in \pi^{-1}(u)$, we have $w=s(u)z(u)$.
Then $\widetilde{\eta}^j_{k,h}(w)=z^2\frac{\pi^*\eta^j_{k,h}}{r^3}(w)$.
So finally we have
\begin{align*}
\widetilde{\eta}
&=\RE(\widetilde{\eta}^j_{k,h}\widetilde{\theta}^k\otimes\overline{\widetilde{\theta}_j}\otimes\widetilde{\theta}^h)
=\RE\left(z^2\frac{\pi^*\eta^j_{k,h}}{r^3}
(r\pi^*\theta^k\otimes r\pi^*\overline{\theta_j}\otimes r\pi^*\theta^h)\right)\\
&=\RE(z^2\pi^*\eta^j_{k,h}
\pi^*\theta^k\otimes \pi^*\overline{\theta_j}\otimes \pi^*\theta^h)
=\RE(z^2\pi^*\eta).
\qedhere
\end{align*}
\end{proof}

\begin{defi}
Given a section $s\colon U\to S$ with $U$ open subset of $M$, we will call the corresponding tensor $\eta$ found in Proposition \ref{prop:ThetatildeWrtTheta} the \emph{deviance tensor} with respect to $s$.
\end{defi}

We can give a more global formulation of Proposition \ref{prop:ThetatildeWrtTheta} in the following terms
\begin{prop}\label{prop:FibratoPrinC}
Given a projective special K\"{a}hler manifold $(\pi\colon \widetilde{M}\to M,\nabla)$, there exists a map $\gamma\colon \widetilde{M}\to \sharp_2 S_{3,0}M\subset T_{1,0}M\otimes T^{0,1}M\otimes T_{1,0}M$ of bundles over $M$, such that $\gamma(ua)=a^2\gamma(u)$ and for every local section $s\colon U\to S\subset \widetilde{M}$, the deviance induced by $s$ is $\eta=\gamma\circ s$.

Let $L:=\widetilde{M}\times_{\C^*} \C$, then $\gamma$ can be identified with a homomorphism of complex vector bundles $\widehat{\gamma}\colon L\otimes L\to \sharp_2 S_{3,0}M$ such that $\gamma (u)=\widehat{\gamma}([u,1]\otimes [u, 1])$.
\end{prop}
\begin{proof}
Let $u\in\widetilde{M}$, then there exists an open neighbourhood $U\subseteq M$ of $u$ and local trivialisation $(\pi|_{\pi^{-1}(U)},z)\colon \pi^{-1}(U)\to U\times\C^*$ induced by a section $s\colon U\to S$ so, for all $w\in \pi^{-1}(U)$ we have $w=s(\pi(w))z(w)$.
Let now $\eta\colon U\to \sharp_2 S_{3,0}M$ be the deviance corresponding to $s$; we define $\gamma(u):=z(u)^2\eta(p)$ where $p=\pi(u)$.
This definition is independent on the choice of $s$.
In order to prove it take another $s'\colon U'\to S$ with $p\in U'$ and the corresponding $z'$ and $\eta'$, then, on $U\cap U'$, there is a map $c:=z\circ s'\colon U\cap U'\to \C$ whose image is in $S^1$, as both $s$ and $s'$ are sections of $S$.
By definition, $s'=s\cdot c$.
Since $sz=s'z'$, $z(u)=z(s'(p)z'(u))=z(s'(p))z'(u)=c(p)z'(u)$, so $z=z'\pi^*c$.
Now, by construction $\RE(z'^2\pi^*\eta')=\widetilde{\eta}=\RE(z^2\pi^*\eta)=\RE(z'^2\pi^*c^2 \pi^*\eta')$, so $\eta'=c^2\eta$.
Thus $z(u)^2\eta(p)=z'(u)^2c(p)^2\eta(p)=z'(u)^2\eta'(p)$ and thus $\gamma$ is well defined.

Moreover, $\gamma(ua)=z(ua)^2\eta(\pi(ua))=z(u)^2 a^2\eta(p)=a^2\gamma(u)$.

We can define the homomorphism $L\otimes L\to \sharp_2S_{3,0}M$ locally: given a section $s\colon U\to S$, we map $[u,w]\otimes [u',w']$ to $z(u)z(u')ww'\cdot \eta^s_p$ where $p=\pi(u)=\pi(u')$.
This map does not depend on the choice of the section as one can see from the relations above, and it is also independent on the representatives chosen of these classes; for the first class for example $z(ua)w=z(u)aw$.

This map commutes with the projections on $M$ and it is $\C$-linear on the fibres, so it is a complex vector bundle map.
\end{proof}
\begin{defi}
We call $\gamma\colon S\to \sharp_2S_{3,0}M$ of Proposition \ref{prop:FibratoPrinC} the \emph{intrinsic deviance} of the projective special K\"{a}hler manifold.
\end{defi}
\begin{rmk}
Given a section $s\colon U\to S$ and the corresponding function $z\in\smooth{\pi^{-1}(U),\C^*}$ such that $sz=\id_{\pi^{-1}(U)}$, we can compute $dz=z(\frac{1}{r}dr+id\vartheta)$, since locally $z=re^{i\vartheta}$.
Notice that $\vartheta$ is not globally defined on $\pi^{-1}(U)$, but $d\vartheta$ and $e^{i\vartheta}$ are.
Moreover,
\begin{equation}\label{eq:coordConnection}
\frac{1}{z}dz=\frac{1}{r}dr+id\vartheta\in\Omega^1(\pi^{-1}(U),\C)
\end{equation}
is a principal connection form, in fact it is equivariant for the action of $\C^*$ as $z(ua)=az(u)$ for all $a\in\C$ and, given a complex number $a$ and its corresponding fundamental vector field $a^{\circ}\in\field{\widetilde{M}}$,
\begin{align*}
\frac{1}{z}dz(a^{\circ})_u
&=\frac{1}{z}dz(\frac{d}{dt}ue^{at}|_{t=0})
=\frac{1}{z(u)}\frac{d}{dt}z(ue^{at})|_{t=0}
=\frac{1}{z(u)}\frac{d}{dt}z(u)e^{at}|_{t=0}
=a.
\end{align*}
\end{rmk}
\begin{rmk}\label{rmk:tau}
A local section $s\colon U\to S$ induces $\tau:=s^*\widetilde{\varphi}=s^*\varphi\in\Omega^1(U)$ such that on $\pi^{-1}(U)$
\begin{equation}
\widetilde{\varphi}=d\vartheta+\pi^*\tau
\end{equation}
and thus on $\pi_{S}^{-1}(U)$:
\begin{equation}
\varphi=d\vartheta|_{S}+\pi_S^*\tau.
\end{equation}

If we consider in fact the form $\widetilde{\varphi}-d\vartheta$, we notice that it is basic, as it can also be seen as the difference of two connection forms on $\pi^{-1}(U)$ (namely \eqref{eq:C*connection} and \eqref{eq:coordConnection}) up to a multiplication by $i$.
Therefore, $\widetilde{\varphi}-d\vartheta=\pi^*\tau$ for some $\tau\in\Omega^1(U)$.
The second equation is simply obtained from the first by restriction to $S\subseteq \widetilde{M}$.
\end{rmk}
\section{Characterisation theorem}
In this section we prove our main theorem, characterising projective special K\"{a}hler manifolds in terms of the deviance.
We start by deriving necessary conditions on the deviance, reflecting the curvature conditions of Proposition \ref{prop:SpezzamentoCurvatura}.
\begin{prop}\label{prop:differentialCondition}
For a projective special K\"{a}hler manifold $(\pi\colon \widetilde{M}\to M,\nabla)$ with $(\widetilde{M},\widetilde{g},\widetilde{I},\widetilde{\omega},\nabla,\xi)$, and a local section $s\colon U\to S$, then the corresponding deviance $\eta$ satisfies
\begin{equation}
d^{LC}\eta=2i\tau\wedge \eta
\end{equation}
where $\tau=s^*\varphi\in\Omega^1(U)$.
\end{prop}
\begin{proof}
Thanks to Proposition \ref{prop:ThetatildeWrtTheta}, we know that there exists $z=re^{i\vartheta}$ and $\eta\in T_{1,0}U\otimes T^{0,1}U\otimes T_{1,0}U$ such that on $\pi^{-1}(U)$ we have $\widetilde{\eta}=\RE(z^2\pi^*\eta)$.

Now we would like to describe $\widetilde{d}^{LC}\widetilde{\eta}$ in terms of $d^{LC}\eta$.
Notice that
\begin{equation}\label{eq:TildedEta1}
\begin{aligned}[b]
\widetilde{d}^{LC}\widetilde{\eta}
&=\widetilde{d}^{LC}\RE(z^2\pi^*\eta)
=\RE(\widetilde{d}^{LC}(z^2\pi^*\eta))
=\RE(2zdz\wedge \pi^*\eta+z^2\widetilde{d}^{LC}\pi^*\eta)\\
&=\RE\left(z^2\left(2(\frac{1}{r}dr+id\vartheta) \wedge \pi^*\eta+\widetilde{d}^{LC}\pi^*\eta\right)\right).
\end{aligned}
\end{equation}

The next step is to compute $\widetilde{d}^{LC}\pi^*\eta$, but since we are using the Levi-Civita connection, it is equivalent to compute $\sharp_2(\widetilde{d}^{LC}\pi^*\sigma)$, where $\sigma=\flat_2 \eta\in S_{3,0}U$.
Let us consider a local coframe $\theta$ in $M$ and the corresponding lifting $\widetilde{\theta}$ as in \eqref{eq:coframe Mtilde}, so that we can denote explicitly $\sigma=\sigma_{k,j,h}\theta^k\otimes \theta^{j}\otimes \theta^{h}$.
We have
\begin{align}
\widetilde{\nabla}^{LC}\pi^*\theta^k
&=\widetilde{\nabla}^{LC}\frac{\widetilde{\theta}^k}{r}
=-\frac{dr}{r^2}\otimes \widetilde{\theta}^k-\frac{1}{r}\left((\widetilde{\omega}^{LC})^{k}_{j}\otimes\widetilde{\theta}^{j}\right)\\
&=-\frac{dr}{r}\otimes \pi^*\theta^k-\frac{1}{r}\left(\sum_{j=1}^{n}\pi^*(\omega^{LC})^{k}_{j}\otimes \widetilde{\theta}^{j}+i\widetilde{\varphi}\otimes\widetilde{\theta}^j+\pi^*\theta^k\otimes \theta^{n+1}\right)\\
&=-\frac{dr}{r}\otimes \pi^*\theta^k-\pi^*\left((\omega^{LC})^{k}_{j}\otimes \theta^{j}\right)-i\widetilde{\varphi}\otimes\pi^*\theta^j-\pi^*\theta^k\otimes \frac{1}{r}\theta^{n+1}\\
&=\pi^*\left(\nabla^{LC} \theta^{k}\right)- \frac{1}{r}\theta^{n+1}\otimes \pi^*\theta^k-\pi^*\theta^k\otimes \frac{1}{r}\theta^{n+1}.
\end{align}
We can now compute the following for $X\in\field{\pi^{-1}(U)}$:
\begin{align*}
\widetilde{\nabla}^{LC}_X&\pi^*\sigma
=\widetilde{\nabla}^{LC}_X\pi^*(\sigma_{k,j,h}\theta^k\otimes \theta^{j}\otimes \theta^{h})
=\widetilde{\nabla}^{LC}_X(\pi^*\sigma_{k,j,h}\pi^*\theta^k\otimes \pi^*\theta^{j}\otimes \pi^*\theta^{h})\\
&=d\pi^*\sigma_{k,j,h}(X)\theta^k\otimes \theta^{j}\otimes \theta^{h}
+\pi^*\sigma_{k,j,h}\left(\widetilde{\nabla}^{LC}_X\pi^*\theta^k\otimes \pi^*\theta^{j}\otimes \pi^*\theta^{h}\right.\\
&\quad\left.+\pi^*\theta^k\otimes \widetilde{\nabla}^{LC}_X\pi^*\theta^{j}\otimes \pi^*\theta^{h}
+\pi^*\theta^k\otimes \pi^*\theta^{j}\otimes \widetilde{\nabla}^{LC}_X\pi^*\theta^{h}\right)\\
&=\pi^*d\sigma_{k,j,h}(X)\theta^k\otimes \theta^{j}\otimes \theta^{h}
+\pi^*\sigma_{k,j,h}\pi^*\left(\nabla^{LC} \theta^{k}\right)_X \otimes \pi^*\theta^{j}\otimes \pi^*\theta^{h}\\
&\quad +\pi^*\sigma_{k,j,h}\pi^*\theta^k\otimes \pi^*\left(\nabla^{LC} \theta^{j}\right)_X\otimes \pi^*\theta^{h}\\
&\quad +\pi^*\sigma_{k,j,h}\pi^*\theta^k\otimes \pi^*\theta^{j}\otimes \pi^*\left(\nabla^{LC} \theta^{j}\right)_X
-\frac{3}{r}\widetilde{\theta}^{n+1}(X)\pi^*\sigma\\
&\quad -\frac{1}{r}\left(\pi^*\sigma_{k,j,h}\pi^*\theta^k(X)\widetilde{\theta}^{n+1}\otimes \pi^*\theta^{j}\otimes \pi^*\theta^{h}\right.\\
&\quad\left.+\pi^*\sigma_{k,j,h}\pi^*\theta^k\otimes \pi^*\theta^{j}(X)\widetilde{\theta}^{n+1}\otimes \pi^*\theta^{h}\right.\\
&\quad\left.+\pi^*\sigma_{k,j,h}\pi^*\theta^k\otimes \pi^*\theta^{j}\otimes \pi^*\theta^{h}(X)\widetilde{\theta}^{n+1}\right)\\
&=\pi^*\left(\nabla^{LC}\sigma\right)_X
-\frac{2}{r}\widetilde{\theta}^{n+1}(X)\pi^*\sigma
 -\frac{1}{r}\widetilde{\theta}^{n+1}(X)\pi^*\sigma
-\frac{1}{r}\widetilde{\theta}^{n+1}\otimes\pi^*\sigma(X,\cdot,\cdot)\\
&\quad -\frac{1}{r}\pi^*\sigma(\cdot,X\otimes\widetilde{\theta}^{n+1},\cdot)
-\frac{1}{r}\pi^*\sigma(\cdot,\cdot,X\otimes\widetilde{\theta}^{n+1}).
\end{align*}
In general then, if $\sigma=\theta^k\otimes\sigma_k$, where $\sigma_k=\sigma_{k,j,h}\theta^j\otimes \theta^h\in S_{2,0}U$, we have by symmetry
\begin{align}
\widetilde{\nabla}^{LC}\pi^*\sigma
&=\pi^*\left(\nabla^{LC}\sigma\right)
-\frac{2}{r}\widetilde{\theta}^{n+1}\otimes\pi^*\sigma
 -\frac{2}{r}((\widetilde{\theta}^{n+1})(\pi^*\theta^k))\otimes\pi^*(\sigma_{k,j,h}\theta^j\otimes\theta^h)\\
&\quad -\frac{2}{r}\left(\pi^*(\sigma_{k,j,h}\theta^k\otimes\theta^j)\otimes((\widetilde{\theta}^{n+1})(\pi^*\theta^h))\right).
\end{align}
Notice in particular that the last two rows are symmetric in the first two indices.
 
In order to compute $\widetilde{d}^{LC}\pi^*\sigma$ we need to antisymmetrise $\widetilde{\nabla}^{LC}\pi^*\sigma$ in the first two indices and multiply by two, so only the first row survives and we get
\begin{equation}
\widetilde{d}^{LC}\pi^*\sigma
=\pi^*(d^{LC}\sigma)-\frac{2}{r}\widetilde{\theta}^{n+1}\wedge \pi^*\sigma,
\end{equation}
and therefore
\begin{equation}
\widetilde{d}^{LC}\pi^*\eta
=\pi^*(d^{LC}\eta)-\frac{2}{r}\widetilde{\theta}^{n+1}\wedge \pi^*\eta.
\end{equation}
Substituting this value in \eqref{eq:TildedEta1}, we obtain
\begin{align}\label{eq:TildedEta2}
\widetilde{d}^{LC}\widetilde{\eta}
&=\RE\left(z^2\left(2(\frac{1}{r}dr+id\vartheta) \wedge \pi^*\eta+\pi^*(d^{LC}\eta)-\frac{2}{r}\widetilde{\theta}^{n+1}\wedge \pi^*\eta\right)\right)\\
&=\RE\left(z^2\left(\pi^*d^{LC}\eta-2i(\widetilde{\varphi}-d\vartheta) \wedge \pi^*\eta\right)\right).
\end{align}
As observed in Remark \ref{rmk:tau}, $\widetilde{\varphi}-d\vartheta=\pi^*\tau$, so we have
\begin{align*}
\widetilde{d}^{LC}\widetilde{\eta}
&=\RE\left(z^2\pi^*\left(d^{LC}\eta-2i\tau \wedge \eta\right)\right).
\end{align*}

From Proposition \ref{prop:SpezzamentoCurvatura}, we know that $\widetilde{d}^{LC}\widetilde{\eta}=0$, and since $\eta\in\Omega^1(U,T_{0,1}\otimes T^{1,0})$, $\eta$ and $\overline{\eta}$ are linearly independent, so this quantity vanishes if and only if $z^2 \pi^*\left(d^{LC}\eta-2i\tau \wedge \eta\right)$ does.
Therefore,
\begin{equation}
d^{LC}\eta-2i\tau \wedge \eta=0,
\end{equation}
ending the proof.
\end{proof}

Let us now look at the final ingredient of the curvature tensor, that is $\frac{1}{2}[\widetilde{\eta}\wedge \widetilde{\eta}]$.
In the setting of Proposition \ref{prop:ThetatildeWrtTheta}, given a section $s\colon U\to S$, and the induced deviance $\eta$, then
\begin{align}
\frac{1}{2}[\widetilde{\eta}\wedge \widetilde{\eta}]
&=\frac{1}{2}[\RE(z^2\pi^*\eta)\wedge \RE(z^2\pi^*\eta)]
=\frac{1}{2}[z^2\pi^*\eta+\overline{z}^2\pi^*\overline{\eta}\wedge z^2\pi^*\eta+\overline{z}^2\pi^*\overline{\eta}]\\
&=\frac{1}{2}\RE\left(z^4[\pi^*\eta\wedge \pi^*\eta]\right)
+|z|^4[\pi^*\eta\wedge \pi^*\overline{\eta}].
\end{align}
We can compute this tensor for a local coframe $\theta$ on $M$.
Since we have
\begin{equation}
\pi^*\theta^k\circ \pi^*\theta_h
=\frac{1}{r}\widetilde{\theta}^k(\frac{1}{r}\widetilde{\theta}_h)
=\frac{1}{r^2}\widetilde{\theta}^k(\widetilde{\theta}_h)
=\frac{1}{r^2}\delta^k_h
=\frac{1}{r^2}\pi^*(\theta^k\circ\theta_h)
\end{equation}
and $\pi^*\theta^k\circ \pi^*\overline{\theta_h}=\pi^*\overline{\theta^k}\circ \pi^*\theta_h=0$, then
\begin{align}
[\pi^*\eta\wedge \pi^*\eta]
&=[\pi^*\eta^{j}_{k,h}\pi^*\theta^{k}\otimes \pi^*\overline{\theta_{j}}\otimes \pi^*\theta^{h}\wedge \pi^*\eta^{j'}_{k',h'}\pi^*\theta^{k'}\otimes \pi^*\overline{\theta_{j'}}\otimes \pi^*\theta^{h'}]\\
&=\pi^*\eta^{j}_{k,h}\pi^*\theta^{k}\wedge \pi^*\eta^{j'}_{k',h'}\pi^*\theta^{k'}\otimes [\pi^*\overline{\theta_{j}}\otimes \pi^*\theta^{h},\pi^*\overline{\theta_{j'}}\otimes \pi^*\theta^{h'}]
=0
\end{align}
and
\begin{align}
[\pi^*\eta\wedge& \pi^*\overline{\eta}]
=[\pi^*\eta^{j}_{k,h}\pi^*\theta^{k}\otimes \pi^*\overline{\theta_{j}}\otimes \pi^*\theta^{h}\wedge \pi^*\overline{\eta^{j'}_{k',h'}}\pi^*\overline{\theta^{k'}}\otimes \pi^*\theta_{j'}\otimes \pi^*\overline{\theta^{h'}}]\\
&=\pi^*\eta^{j}_{k,h}\pi^*\theta^{k}\wedge \pi^*\overline{\eta^{j'}_{k',h'}}\pi^*\overline{\theta^{k'}}\otimes [\pi^*\overline{\theta_{j}}\otimes \pi^*\theta^{h},\pi^*\theta_{j'}\otimes \pi^*\overline{\theta^{h'}}]\\
&=\pi^*(\eta^{j}_{k,h}\theta^{k}\wedge \overline{\eta^{j'}_{k',h'}}\overline{\theta^{k'}})\otimes \frac{1}{r^2}\pi^*(\overline{\theta_{j}}\otimes \theta^{h}(\theta_{j'})\otimes \overline{\theta^{h'}}-\theta_{j'}\otimes \overline{\theta^{h'}}(\overline{\theta_{j}})\otimes \theta^{h})\\
&=\frac{1}{r^2}\pi^*[\eta\wedge \overline{\eta}].
\end{align}
Therefore
\begin{equation}\label{eq:etawedgeeta}
\frac{1}{2}[\widetilde{\eta}\wedge\widetilde{\eta}]
=\frac{|z|^4}{r^2}\pi^*[\eta\wedge\overline{\eta}]
=r^2\pi^*[\eta\wedge\overline{\eta}].
\end{equation}
\begin{rmk}\label{rmk:sezioneGlobaleH}
Note that $[\eta\wedge\overline{\eta}]$ is independent on the local coframe, and if we consider another section such that $s'=sa$ on the intersection of their domains, with $a$ taking values in $S^1$, if $\eta'$ is the deviance corresponding to $s'$, then $[\eta'\wedge \overline{\eta'}]=[\eta a\wedge \overline{\eta}\overline{a}]=|a|^2[\eta \wedge \overline{\eta}]=[\eta \wedge \overline{\eta}]$.
So, there is a globally defined section $M\to S^2(\lie{u}(n))$ mapping $p$ to $[\eta_p\wedge \overline{\eta_p}]$.
\end{rmk}

For a projective special K\"{a}hler manifold $(\pi\colon \widetilde{M}\to M,\nabla)$ of real dimension $2n$, Proposition \ref{prop:SpezzamentoCurvatura}, interpreted in the light of the last observations and the ones made in Section \ref{sec:coframe} (see Remark \ref{rmk:Levi-Civita_CSK}), says that $0=r^2\pi^*(\Omega^{LC}+\Omega_{\Pj_{\C}^n}+[\eta\wedge \overline{\eta}])$, thus we have the following equation:
\begin{equation}\label{eq:curvatureEquation}
\Omega^{LC}+\Omega_{\Pj_{\C}^n}+[\eta\wedge \overline{\eta}]=0.
\end{equation}
This is a curvature tensor, so we can compute its Ricci and scalar component.
\begin{prop}
Let $(\pi\colon \widetilde{M}\to M,\nabla)$ be a projective special K\"{a}hler manifold of dimension $2n$, then
\begin{equation}\label{eq:RicciPSK}
\Ric_M(X,Y)+2(n+1)g(X,Y)-\RE(h(\overline{\eta_X},\eta_Y))=0;
\end{equation}
\begin{equation}\label{eq:ScalCurv}
\scal_M+2(n+1)-\frac{2}{n}\norm{\eta}_h^2=0.
\end{equation}
\end{prop}
\begin{proof}
The first summand in \eqref{eq:curvatureEquation} gives the Ricci tensor of $M$, the second gives the Ricci tensor of the projective space \eqref{eq:RicciProj}.
In order to compute the last term, consider a unitary frame $\theta$; from previous computations,
\begin{align*}
[\eta\wedge \overline{\eta}]
&=(\eta^{j}_{k,h}\theta^{k}\wedge \overline{\eta^{j'}_{k',h'}}\overline{\theta^{k'}})\otimes (\delta^{h}_{j'}\overline{\theta_{j}}\otimes \overline{\theta^{h'}}-\delta^{h'}_{j}\theta_{j'}\otimes \theta^{h})\\
&=\RE\left(\eta^{j}_{k,h}\overline{\eta^{h}_{k',h'}}\theta^{k}\wedge \overline{\theta^{k'}}\otimes \overline{\theta_{j}}\otimes \overline{\theta^{h'}}\right)
\end{align*}
then the Ricci component $\Ric([\eta\wedge\overline{\eta}])$ evaluated on $X=\RE(X^k\theta_k)$ and $Y=\RE(Y^k\theta_k)$ is the trace of $[\eta\wedge\overline{\eta}](\cdot,Y)X$, which is
\begin{align*}
[\eta\wedge&\overline{\eta}](\cdot,Y)X\\
&=\eta^{j}_{k,h}\overline{\eta^{h}_{u,v}}(\theta^{k} \overline{Y^{u}}-Y^k \overline{\theta^{u}})\otimes \overline{\theta_{j}}\otimes \overline{X^{v}}+
\overline{\eta^{j}_{k,h}}\eta^{h}_{u,v}(\overline{\theta^{k}} Y^{u}-\overline{Y^{u}}\theta^{k})\otimes \theta_{j}\otimes X^{v}\\
&=\RE\left(\eta^{j}_{k,h}\overline{\eta^{h}_{u,v}}(\theta^{k} \overline{Y^{u}}-Y^k \overline{\theta^{u}})\otimes \overline{\theta_{j}}\otimes \overline{X^{v}}\right).
\end{align*}
Its trace is therefore 
\begin{equation}
-\RE\left(\eta^{j}_{k,h}\overline{\eta^{h}_{j,v}}Y^k \overline{X^{v}}\right)=-\RE\left(\eta^{j}_{k,h}\overline{\eta^{h}_{u,j}}Y^k \overline{X^{u}}\right)=-\RE(h(\overline{\eta_X},\eta_Y)),
\end{equation}
or equivalently, $\Ric([\eta\wedge\overline{\eta}])=-\RE\left(\overline{\eta^{h}_{u,j}}\eta^{j}_{k,h}\overline{\theta^{u}}\theta^k \right)$.
Thus we obtain \eqref{eq:RicciPSK}.

From this tensor we can now obtain \eqref{eq:ScalCurv} by computing the scalar component, that is by taking the trace, raising the indices with $g$ and then dividing it by the dimension of $M$.
Thus the first summand gives $\scal_M$, the second gives $2(n+1)$ and the third
\begin{align}
\frac{1}{2n}\tr\left(-\RE\left(\overline{\eta^{h}_{u,j}}\eta^{j}_{k,h}(\overline{\theta^{u}})_\sharp\theta^k \right) \right)
&=-\frac{1}{2n}\tr\left(\RE\left(\overline{\eta^{h}_{u,j}}\eta^{j}_{k,h}(2\theta_{u})\theta^k \right)\right)\\*
&=-\frac{1}{n}\sum_{j,h,k} \RE\left(\eta^{j}_{k,h}\overline{\eta^{h}_{k,j}}\right)=-\frac{2}{n}\norm{\eta}_h^2.\qedhere
\end{align}
\end{proof}
In particular, since the norm of $\eta$ is non negative, we obtain a lower bound for the scalar curvature:
\begin{cor}
Let $(\pi\colon \widetilde{M}\to M,\nabla)$ be a projective special K\"{a}hler manifold, then
\begin{equation}
\scal_M\ge -2(n+1).
\end{equation}
Equality holds at a point if and only if the deviance vanishes at that point.
\end{cor}
\begin{rmk}
The lower bound is reached by projective special K\"{a}hler manifolds with zero deviance; we will see that this condition characterises the complex hyperbolic space (Proposition \ref{prop:uniquenessHyp}).
\end{rmk}

We can now state the main result:
\begin{theo}\label{theo:characterisationPSK}
On a $2n$-dimensional K\"{a}hler manifold $(M,g,I,\omega)$, to give a projective special K\"{a}hler structure is equivalent to give an $S^1$-bundle $\pi_S\colon S\to M$ endowed with a connection form $\varphi$ and a bundle map $\gamma\colon S\to \sharp_2 S_{3,0} M$ such that:
\begin{enumerate}
\item\label{theo:charPSKPunto1} $d\varphi=-2\pi_S^*\omega$;
\item\label{theo:charPSKPunto2} $\gamma(u a)=a^2\gamma(u)$ for all $a\in S^1$;
\item\label{theo:charPSKPunto3} for a certain choice of an open covering $\{U_\alpha|\alpha\in\mathcal{A}\}$ of $M$ and a family $\{s_\alpha\colon U_\alpha\to S\}_{\alpha\in\mathcal{A}}$ of sections, denoting by $\eta_\alpha$ the local $1$-form taking values in $T^{0,1}M\otimes T_{1,0}M$ determined by $\gamma\circ s_{\alpha}$, for all $\alpha\in\mathcal{A}$:
\begin{enumerate} [label=\textbf{D\arabic*}]
\item\label{theo:charpsk:CurvCond} $\qquad\Omega^{LC}+\Omega_{\Pj^n_{\C}}+[\eta_\alpha\wedge\overline{\eta_\alpha}]=0$;
\item\label{theo:charpsk:DiffCond} $\qquad d^{LC}\eta_\alpha=2i s_{\alpha}^*\varphi\wedge\eta_\alpha$.
\end{enumerate}
\end{enumerate}
In this case, \ref{theo:charPSKPunto3} is satisfied by every such family of sections.
\end{theo}
\begin{proof}
Given a projective special K\"{a}hler manifold, we define $S:=r^{-1}(1)\subset \widetilde{M}$ and $\varphi:=-\iota_\xi\omega|_{S}$.
The principal action on $S$ is generated by $I\xi$ which is tangent to $S$ since $T_uS=\ker(dr)$ and $dr(I\xi)=-\frac{1}{r}\xi^\flat (I\xi)=-\frac{\widetilde{g}(\xi,I\xi)}{r}$.
The curvature is then $d\varphi=-2\pi_S^*\omega$ as shown in Remark \ref{rmk:curvatureSbundle}, so the first point is satisfied.
The second condition holds thanks to Proposition \ref{prop:FibratoPrinC}.
For the third point, we get \ref{theo:charpsk:CurvCond} from the arguments leading to equation \eqref{eq:curvatureEquation} and \ref{theo:charpsk:DiffCond} from Proposition \ref{prop:differentialCondition}.

In order to prove the other direction, define $\widetilde{M}:=S\times \R^+$, $\pi:=\pi_S\circ \pi_1\colon \widetilde{M}\to M$, and $t:=\pi_2\in\smooth{\widetilde{M},\R^+}$, where $\pi_1\colon S\times \R^+\to S$ and $\pi_2\colon S\times \R^+\to \R^+$ are the projections.
Let $\widetilde{\varphi}:=\pi_1^*\varphi$, in particular $d\widetilde{\varphi}=\pi_1^*d\varphi=-2\pi^*\omega$ as expected.
Define now
\begin{equation}
\widetilde{g}:=t^2\pi^*g-t^2\widetilde{\varphi}^2-dt^2
\end{equation}
which is non-degenerate, since $r\widetilde{\varphi}$ and $dt$ are linearly independent and transverse to $\pi$, so we can form a basis for the $1$-forms according to which we can see that $\widetilde{g}$ has signature $(2n,2)$.
Extend now $I$ to $\widetilde{I}$ so that $\widetilde{I}\cdot(\pi^*\alpha)=\pi^*I\alpha$ for all $\alpha\in T^*M$ and $\widetilde{I}\cdot(dt)=t\widetilde{\varphi}$.

The metric $\widetilde{g}$ is compatible with $\widetilde{I}$ since $\widetilde{I}\cdot \widetilde{g}=t^2\widetilde{I}\cdot\pi^*g-(\widetilde{I}\cdot t\widetilde{\varphi})^2-(\widetilde{I}\cdot dt)^2=t^2\pi^*(I\cdot g)-(-dt)^2-(t\widetilde{\varphi})^2=t^2\pi^*(I\cdot g)-dt^2-t^2\widetilde{\varphi}^2=\widetilde{g}$.

We thus have a K\"{a}hler manifold $(\widetilde{M},\widetilde{g},\widetilde{I},\widetilde{\omega})$, where 
\begin{equation}
\widetilde{\omega}:=t^2\pi^*\omega+t\widetilde{\varphi}\wedge dt.
\end{equation}

Let $\xi:=t\partial_t$ where $\partial_t$ is the vector field corresponding to the coordinate derivation on $\R^+$.
Notice that the function $r=\sqrt{-\widetilde{g}(\xi,\xi)}$ coincides with $t$, as $\sqrt{-\widetilde{g}(t\partial_t,t\partial_t)}=\sqrt{-t^2 \widetilde{g}(\partial_t,\partial_t)}=t$.
In particular $\widetilde{g}(\xi,\xi)=-t^2\neq 0$ and $\widetilde{g}(\widetilde{I}\xi,\widetilde{I}\xi)=\widetilde{g}(\xi,\xi)<0$, so $\widetilde{g}$ is negative definite on $\langle\xi,I\xi\rangle$ and hence positive definite on the orthogonal complement.

Let now $\theta$ be a unitary coframe on an open subset $U\subseteq M$, then we can lift it to a complex coframe $\widetilde{\theta}$ on $\pi^{-1}(U)$ defined as in \eqref{eq:coframe Mtilde}.
It is straightforward to check that $\widetilde{\theta}$ is adapted to the pseudo-K\"{a}hler structure of $\widetilde{M}$.
Notice that the proof of Proposition \ref{prop:conicLC} is still valid in this situation even though we do not know whether $\widetilde{M}\to M$ has a structure of projective special K\"{a}hler manifold; this gives us a description of the Levi-Civita connection form on $\widetilde{M}$ with respect to $\widetilde{\theta}$.
Notice that $\widetilde{\theta}^k(\xi)=0$ for $k\le n$ and $\widetilde{\theta}^{n+1}(\xi)=dt(t\partial_t)+i\widetilde{\varphi}(t\partial_t)=t$ so $\xi=\RE(t\widetilde{\theta}_{n+1})$.
We can thus compute
\begin{align}
\widetilde{\nabla}^{LC}\xi
&=dt\otimes \RE(\widetilde{\theta}_{n+1})+t\widetilde{\nabla}^{LC}\RE(\widetilde{\theta}_{n+1})\\
&=\RE(dt\otimes \widetilde{\theta}_{n+1})+\frac{t}{r}\RE\left(\sum_{k=1}^n \widetilde{\theta}^k\otimes\widetilde{\theta}_k+ i\Im(\widetilde{\theta}^{n+1})\otimes\widetilde{\theta}_{n+1}\right)\\
&=\RE\left(\sum_{k=1}^{n+1} \widetilde{\theta}^k\otimes\widetilde{\theta}_k\right)
=\id.
\end{align}

Each section $s_\alpha$ corresponds to the trivialisation $(\pi|_{\pi^{-1}(U)},z_{\alpha})\colon \pi^{-1}U\to U\times \C^*$ in the sense that $s(\pi(u))\cdot z_{\alpha}(u)=u$ for all $u\in\pi^{-1}(U_\alpha)$.
For all $\alpha$ on $\pi^{-1}(U_{\alpha})$, define the tensor $\widetilde{\eta}_\alpha:=\RE(z_\alpha^2\pi^*\eta_\alpha)$.
The family $\{\widetilde{\eta}_\alpha\}_{\alpha\in\mathcal{A}}$ is compatible on intersections $U_1\cap U_2$, in fact if $s_1=cs_2$ for $c\in \U(1)$, then $z_2=cz_1$ and $\eta_1=\gamma\circ s_1=\gamma\circ cs_2=c^2\gamma\circ s_2=c^2\eta_2$, so
\begin{equation}
\widetilde{\eta_1}
=\RE(z_1^2\pi^*\eta_1)
=\RE(z_1^2c^2\pi^*\eta_2)
=\RE(z_2^2\pi^*\eta_2)
=\widetilde{\eta_2}.
\end{equation}
Therefore, this family glues to form a tensor $\widetilde{\eta}\in\sharp_2 S^3\widetilde{M}$.

We can build another connection $\nabla:=\widetilde{\nabla}^{LC}+\widetilde{\eta}$.
Notice that $\nabla\xi=\widetilde{\nabla}^{LC}\xi+\widetilde{\eta}(\xi)=\id+\RE(z_{\alpha}^2\pi^*\eta_{\alpha})(\xi)=\id$ because locally $\eta_\alpha$ is horizontal for all $\alpha$.

In order to prove that $\nabla$ is symplectic, since the Levi-Civita connection is symplectic, it is enough to prove that $\widetilde{\omega}(\widetilde{\eta},\cdot)+\widetilde{\omega}(\cdot,\widetilde{\eta})=0$.
Locally, $\widetilde{\omega}=\frac{1}{2i}\sum_{k=1}^{n+1}\overline{\widetilde{\theta}^k}\wedge \widetilde{\theta}^k$ and in fact, for all $X=\RE(X^k\widetilde{\theta}_k)$, $Y=\RE(Y^k\widetilde{\theta}_k)$, $Z=\RE(Z^k\widetilde{\theta}_k)$ vector fields on $\widetilde{M}$:
\begin{align*}
2i(\widetilde{\omega}(\widetilde{\eta}_X Y,Z)+\widetilde{\omega}(Y,\widetilde{\eta}_X Z))
&=\sum_{k=1}^{n+1}\left(\overline{\widetilde{\theta}^k}(\widetilde{\eta}_X Y)\widetilde{\theta}^k(Z)
-\widetilde{\theta}^k(\widetilde{\eta}_X Y)\overline{\widetilde{\theta}^k}(Z)\right.\\
&\quad\left.+\overline{\widetilde{\theta}^k}(Y)\wedge \widetilde{\theta}^k(\widetilde{\eta}_X Z)
-\widetilde{\theta}^k(Y)\wedge \overline{\widetilde{\theta}^k}(\widetilde{\eta}_X Z)\right)\\
&=\sum_{k=1}^{n+1}\left(
z \pi^*\eta^k_{u,v}X^u Y^v Z^k
-\overline{Z^k}\overline{z}^2\overline{\pi^*\eta}^k_{u,v}\overline{X^u} \overline{Y^v} \right.\\
&\quad +\left.\overline{Y}^k\overline{z}^2\overline{\pi^*\eta}^k_{u,v}\overline{X^u} \overline{Z^v}
-z^2\pi^*\eta^k_{u,v}X^u Z^v Y^k
\right)\\
&=\sum_{k=1}^{n+1}\RE\left(
z^2 \pi^*\eta^k_{u,v}X^u Y^v Z^k
-z^2\pi^*\eta^k_{u,v}X^u Z^v Y^k
\right)\\
&=\sum_{k=1}^{n+1}\RE\left(
z^2 \pi^*(\eta^k_{u,v}-\eta^v_{u,k})X^u Y^v Z^k
\right).
\end{align*}
By the symmetry of $\eta$, this quantity vanishes.

Proving that $d^{\nabla}\widetilde{I}=0$, is equivalent to proving that $\nabla \widetilde{I}$ is symmetric in the two covariant indices, and thus $\nabla \widetilde{I}=\widetilde{\nabla}^{LC}\widetilde{I}+[\widetilde{\eta},\widetilde{I}]=[\widetilde{\eta},\widetilde{I}]$.
Since $\widetilde{I}=\RE(i\widetilde{\theta}_k\widetilde{\theta}^k)$, we have
\begin{align*}
[\widetilde{\eta},\widetilde{I}]
&=iz^2\pi^*\eta^u_{v,w}\widetilde{\theta}^v\otimes\overline{\widetilde{\theta}_u}\otimes\widetilde{\theta}^w
-i\overline{z^2\pi^*\eta^u_{v,w}\widetilde{\theta}^v\otimes\overline{\widetilde{\theta}_u}\otimes\widetilde{\theta}^w}\\
&\quad +iz^2\pi^*\eta^u_{v,w}\widetilde{\theta}^v\otimes\overline{\widetilde{\theta}_u}\otimes\widetilde{\theta}^w
-i\overline{z^2\pi^*\eta^u_{v,w}\widetilde{\theta}^v\otimes\overline{\widetilde{\theta}_u}\otimes\widetilde{\theta}^w}
=2i\widetilde{\eta}=-2\widetilde{I}\widetilde{\eta},
\end{align*}
which is symmetric, proving $d^{\nabla}I=0$.

For the flatness of $\nabla$, we compute the curvature locally
\begin{align*}
\Omega^\nabla
=d\omega^{\nabla}+\frac{1}{2}[\omega^{\nabla}\wedge\omega^{\nabla}]
=\widetilde{\Omega}^{LC}+\widetilde{d}^{LC}\widetilde{\eta}+\frac{1}{2}[\widetilde{\eta}\wedge\widetilde{\eta}].
\end{align*}
By Proposition \ref{prop:conicLC}, $\widetilde{\Omega}^{LC}=r^2\pi^*(\Omega^{LC}+\Omega_{\Pj^n_{\C}})$.
For the same reasoning exposed in the proof of Proposition \ref{prop:differentialCondition}, $\widetilde{d}^{LC}\widetilde{\eta}=0$ if and only if $d^{LC}\eta-2is^*\varphi\wedge \eta=0$, which is granted by \ref{theo:charpsk:DiffCond}.

Finally, the computations leading to equation \eqref{eq:etawedgeeta} still apply and thus we can deduce that
\begin{equation}
\Omega^{\nabla}=r\pi^*(\Omega^{LC}+\Omega_{\Pj^n_\C}+[\eta\wedge\overline{\eta}])=0,
\end{equation}
making the connection $\nabla$ flat.

Notice that $\pi\colon \widetilde{M}\to M$ is a principal $\C^*$-bundle, where for all $l e^{i\theta}\in\C^*$ and $(u,t)\in\widetilde{M}$:
\begin{equation}
(u,t)l e^{i\theta}:=(u e^{i\theta},tl).
\end{equation}
The infinitesimal vector field corresponding to $1$ at $(u,t_0)$ is $\xi_{(u,t_0)}$ and the one corresponding to $i$ is $X:=\frac{d}{dt}((u,t_0)\exp(it))|_{t=0}=\frac{d}{dt}(ue^{it},t_0)|_{t=0}$, which is vertical and such that $\widetilde{\varphi}(X)=\varphi(p_*X)=\varphi(\frac{d}{dt}(ue^{it})|_{t=0})=1$ and $dr(X)=0$.
This means that $X=I\xi$ since $\widetilde{g}(X,\cdot)=-r^2\widetilde{\varphi}=-rIdr=I\xi^{\flat}$.

We are only left to prove that $M$ is the K\"{a}hler quotient or $\widetilde{M}$ with respect to the $\U(1)$-action and in order to do so, notice that $\widetilde{\omega}(I\xi,\cdot)=-\widetilde{g}(\xi,\cdot)=rdr=d\left(\frac{r^2}{2}\right)$, so $\mu:=\frac{r^2}{2}$ is a moment map for $I\xi$.
Notice that $\mu^{-1}(\frac{1}{2})=S\times\{1\}$ and $S$ is a principal bundle so, by definition of $\widetilde{g}$ and $\widetilde{\omega}$, $S/\U(1)$ is isometric to $M$ and this ends the proof.
\end{proof}
\begin{rmk}
Starting from the family $\{\eta_\alpha\}_{\alpha}$, we can build a bundle map $\gamma\colon S\to M$ as long as the $\eta_\alpha$'s are linked by the relation $\eta_\alpha=g_{\alpha,\beta}^2\eta_\beta$ where $g_{\alpha,\beta}$ is a cocycle defining $S$.
\end{rmk}

\begin{rmk}\label{rmk:trivial_cohomology2=trivial_line_bundles}
Let $(M,g,I)$ be a K\"{a}hler manifold, then if $H^2(M,\Z)=0$, in particular, every complex line bundle and every circle bundle are trivial.
Moreover, by de Rham's theorem, $H^2_{dR}(M)=H^2(M,\R)=H^2(M,\Z)\otimes\R=0$, so in particular $\omega=d\lambda$ for some $\lambda\in\Omega^1(M)$.
\end{rmk}
\begin{cor}\label{cor:charPSKesatte}
A K\"{a}hler $2n$-manifold $(M,g,I,\omega)$ such that $H^2(M,\Z)=0$, has a projective special K\"{a}hler structure if and only if there exists a section $\eta\colon M\to \sharp_2 S_{3,0}M$ such that
\begin{enumerate} [label=\textbf{D\arabic*\textsuperscript{*}}]
\item\label{eq:condCurvCor} $\qquad \Omega^{LC}+\Omega_{\Pj^n_{\C}}+[\eta\wedge\overline{\eta}]=0;$
\item\label{eq:condDiffCor} $\qquad d^{LC}\eta=-4i\lambda\wedge \eta$;
\end{enumerate}
for some $\lambda\in\Omega^{1}(M)$ such that $d\lambda=\omega$.
\end{cor}
\begin{proof}
If $M$ has a projective special K\"{a}hler structure, then from Theorem \ref{theo:characterisationPSK} we obtain an $S^1$-bundle $p\colon S\to M$ and the map $\gamma\colon S\to \sharp_2 S_{3,0}M$.
Consider the corresponding line bundle $L=S\times_{\U(1)}\C$.
As noted in Remark \ref{rmk:trivial_cohomology2=trivial_line_bundles}, we can assume $L=M\times \C$ and $S=M\times S^1$.
In particular, there is a global section $s\colon M\to S$ and if we call $\eta=\gamma\circ s\colon M\to \sharp_2S_{3,0}M$, it is a global section satisfying the curvature equation thanks to Theorem \ref{theo:characterisationPSK}.
Defining $\lambda:=-\frac{1}{2}s^*\varphi$, we have $d\lambda=-\frac{1}{2}s^*(-2\pi_S^*\omega)=(\pi_S s)^* \omega=\omega$ and thus also the differential condition is satisfied by Theorem \ref{theo:characterisationPSK}.

Conversely, by de Rham's Theorem, we have $\lambda\in\Omega^1(M)$ such that $d\lambda=\omega$.
We define $\pi_S=\pi_1\colon S=M\times S^1\to M$ and choose as connection the form $\varphi=\pi_2^*d\vartheta-2\pi_S^*\lambda$, where $d\vartheta$ is the fundamental $1$-form on $S^1=\U(1)$.
Then $d\varphi=0-2\pi_S^*d\lambda=-2\pi_S^*\omega$, so $S\to M$ has the desired curvature.
Moreover, it is trivial, so we have a global section $s\colon M\to S$ mapping $p$ to $(p,1)$.

Given $\eta\colon M\to \sharp_2S_{3,0}M$ as in the statement, we define $\gamma\colon S\to \sharp_2 S_{3,0}M$ such that $\gamma(p,a):=a^2\eta(p)$ for all $p\in M$ and $a\in \U(1)$.
Notice that $\gamma\circ s=\gamma(\cdot,1)=\eta$, so the curvature equation of this corollary gives the curvature equation in Theorem \ref{theo:characterisationPSK} and the same is true for the differential condition, since $s^*\varphi=s^*\pi_2^*d\vartheta-2s^*\pi_S^*\lambda=0-2\lambda$.
By Theorem \ref{theo:characterisationPSK}, $M$ is thus projective special K\"{a}hler.
\end{proof}
\begin{rmk}
Instead of requiring a section $\eta$ as in Corollary \ref{cor:charPSKesatte}, we could use a section $\sigma$ of $S_{3,0}M$ such that $\sharp_2\sigma=\eta$.
\end{rmk}
\section{Varying the projective special K\"{a}hler structure by a \texorpdfstring{$\U(1)$}{TEXT}-valued function}
Theorem \ref{theo:characterisationPSK} allows to find a whole class of projective special K\"{a}hler structures from a given one, as shown in the following
\begin{prop}\label{prop:PSKcerchioConnessioni}
Let $(\pi\colon \widetilde{M}\to M,\nabla)$ be a projective special K\"{a}hler manifold, let $\gamma\colon S\to \sharp_2 S_{3,0}M$ be its intrinsic deviance and $\varphi\in\Omega^1(S)$ the principal connection form on $\pi_S\colon S\to M$, then for all $\beta\in\smooth{M,\U(1)}$ there is a new projective special K\"{a}hler structure $(\pi\colon \widetilde{M}^{\beta}\to M,\nabla^{\beta})$ with intrinsic deviance $\gamma^\beta=\beta\gamma\colon S\to\sharp_2 S_{3,0}M$ on the same bundle $S$, with principal connection form $\varphi^{\beta}=\pi_S^*\left(\frac{d\beta}{2i\beta}\right)+\varphi$.
\end{prop}
\begin{proof}
We want to use Theorem \ref{theo:characterisationPSK}, so consider the same bundle $\pi_S\colon S\to M$, but with the new connection form $\varphi^\beta$.
Notice that $\varphi^\beta$ is a real form, in fact $\overline{\beta}\beta=1$, so 
\begin{equation}
0=\beta d\overline{\beta}+\overline{\beta}d\beta
=\overline{\beta}\beta \left(\frac{d\overline{\beta}}{\overline{\beta}}+\frac{d\beta}{\beta}\right)
=\left(\overline{\left(\frac{d\beta}{\beta}\right)}+\frac{d\beta}{\beta}\right)
=2\Re\left(\frac{d\beta}{\beta}\right),
\end{equation}
and thus $\Im\left(\frac{d\beta}{2i\beta}\right)=-\frac{1}{2}\Re\left(\frac{d\beta}{\beta}\right)=0$.
Moreover $d\varphi^\beta=-\pi_S^*\left(\frac{d\beta\wedge d\beta}{\beta^2}\right)+d\varphi=d\varphi=-2\pi^*\omega$, so condition \ref{theo:charPSKPunto1} of Theorem \ref{theo:characterisationPSK} this is a compatible principal connection form.
The bundle map $\gamma^\beta$ is still homogeneous of degree $2$.
We are only left to prove the two conditions of point \ref{theo:charPSKPunto3}, so consider a family of sections $\{(U_\alpha,s_\alpha)\}_{\alpha\in\mathcal{A}}$ corresponding to a trivialisation of $S$ and let $\eta^\beta_\alpha:=\gamma^\beta\circ s_\alpha=\beta\gamma\circ s_{\alpha}=\beta \eta_\alpha$.
We thus have
\begin{align*}
d^{LC}\eta^\beta_\alpha
&=d^{LC}(\beta\eta_\alpha)
=d\beta\wedge \eta_\alpha+\beta 2i s_\alpha^*\varphi\wedge \eta_\alpha
=2i\left(\frac{d\beta}{2i\beta}+ s_\alpha^*\varphi\right)\wedge e^{2i\beta}\eta_\alpha\\
&=2is_\alpha^*\left(d\pi_S^*\left(\frac{d\beta}{2i\beta}\right)+ s_\alpha^*\varphi\right)\wedge \eta^\beta_\alpha
=2is_\alpha^*\varphi^\beta\wedge \eta^\beta_\alpha.
\end{align*}
As for the curvature condition \ref{theo:charpsk:CurvCond}, it still holds because 
\begin{equation*}
[\eta_\alpha^\beta\wedge\overline{\eta_\alpha^\beta}]=[\beta\eta_\alpha\wedge \overline{\beta\eta_\alpha}]=[\eta_\alpha\wedge\overline{\eta_\alpha}].
\qedhere
\end{equation*}
\end{proof}

These modified deviances do not always provide an entirely new projective special K\"{a}hler structure.
Before entering into the details, we recall the following elementary result.
\begin{lemma}\label{lemma:differential_action}
Let $M$ be a smooth manifold and $G$ a Lie group with Lie algebra $\lie{g}$ such that there is a smooth right action
\begin{equation}
r\colon M\times G\longrightarrow M.
\end{equation}
Then, the differential of $r$ at a point $(x,a)$ is
\begin{equation}
r_*(X,A)=(R_a)_*(X)+A^{\circ},
\end{equation}
for all $X\in T_xM$, $A\in\lie{g}$, where $A^{\circ}$ denotes the fundamental vector field associated to $A$.
\end{lemma}
\begin{proof}
See e.g.\ \cite[Ex.\ 27.4, p.\ 326]{TuDG}.
\end{proof}
We now present the following isomorphism result:
\begin{prop}\label{prop:isomorphism_if_root}
In the setting of Proposition \ref{prop:PSKcerchioConnessioni}, if moreover $\beta$ has a square root, meaning that $\beta=b^2$ for some $b\colon M\to \U(1)$, then the map
\begin{align}
m_b\colon S&\longrightarrow S,\qquad
u\longmapsto u.b(\pi_S(u))=R_{b(\pi_S(u))}(u)
\end{align}
induces a bundle isomorphism preserving connection and deviance, that is
\begin{equation}
\varphi^\beta=m_b^*(\varphi),\qquad\qquad \gamma^{\beta}=\gamma\circ m_b.
\end{equation}

In particular, if $\beta^*\colon \R\cong H^1_{dR}(S^1)\to H^1_{dR}(M)$ is the zero map, then $\beta$ has a square root.
\end{prop}
\begin{proof}
The preservation of $\gamma$ follows from its $2$-homogeneity, since for all $u\in S$:
\begin{equation}
\gamma\circ m_b (u)=\gamma(u b(\pi_S(u)))
=b(\pi_S(u))^2\gamma(u)
=(\beta\circ \pi_S) \gamma(u)
=\gamma^{\beta}.
\end{equation}

For the connection, we first compute the differential of $m_b$.
Let $r\colon S\times \U(1)\to S$ be the principal right action, then we can see $m_b$ as $r\circ (\id_S\times (b\circ\pi_S))$.
The differential of $(\id_S\times (u\circ\pi_S))$ is $\id_{TS}\times \pi_S^*db$, where $db$ has values in $\lie{u}(1)=i\R$.
Lemma \ref{lemma:differential_action} gives us the differential of the action.
We have
\begin{equation}
((m_b)_*)_u=(R_{b\pi_S(u)})_*+(d_{\pi_S(u)}b)^{\circ}.
\end{equation}
Now let us compute the pullback of $\varphi$, using the fact that $\varphi$ is right invariant and $d\beta=db^2=2bdb$
\begin{align}
m_b^*(\varphi)
&=\varphi\circ (m_b)_*
=\varphi\circ (R_{b\pi_S(u)})_*+\varphi((d_{\pi_S(u)}b)^{\circ})
=R_{b\pi_S(u)}^*\varphi+\frac{1}{ib}d_{\pi_S(u)}b\\
&=\varphi+\frac{1}{i2b^2}d_{\pi_S(u)}\beta
=\varphi+\frac{1}{i2\beta}d_{\pi_S(u)}\beta
=\varphi^{\beta}.
\end{align}

In order to prove the last statement, let $a\colon\U(1)\to\C$ be the standard identification of $\U(1)$ with the unit circle.
Denote by $\psi$ the fundamental form of $\U(1)$, then we can write
\begin{equation}
\psi=\frac{1}{ia}da.
\end{equation}

Now let $\beta\colon M\to \U(1)$, and consider the pullback 
\begin{equation}
\beta^*\psi=\beta^*\bigg(\frac{1}{ia}da\bigg)=\frac{1}{i\beta}d\beta.
\end{equation}
We have $0=\beta^*\colon H^1_{dR}(\U(1))\to H^1_{dR}(M)$, so in particular $\frac{1}{i\beta}d\beta$ is exact.
Let $\lambda\in\smooth{M}$ be such that $d\lambda=\frac{1}{i\beta}d\beta$, then $e^{-i\lambda}\beta$ is a smooth function with image in $\U(1)$ and differential
\begin{equation}
-ie^{i\lambda}\beta d\lambda+e^{i\lambda}d\beta
=-\frac{ie^{i\lambda}\beta}{i\beta} d\beta+e^{i\lambda}d\beta
=-e^{i\lambda} d\beta+e^{i\lambda}d\beta
=0.
\end{equation}
So up to a locally constant function $k$, we have $\beta=ke^{i\lambda}$.
Without loss of generality, we can assume $k=1$ (take $\lambda'=\lambda-i\log(k)$).
Then let $b=e^{\frac{i\lambda}{2}}$ and $b^2=\beta$.
\end{proof}
\begin{rmk}\label{rmk:dR0_unique_structure}
In the family of projective special K\"{a}hler structures constructed in Proposition \ref{prop:PSKcerchioConnessioni}, if $H^1_{dR}(M)=0$, then there is a unique projective special K\"{a}hler structure on $M$ up to a natural notion of isomorphism.
\end{rmk}

Even when $H^1_{dR}(M)\neq 0$, we can say exactly when a function has a global square root by considering the following functional for all $p\in M$:
\begin{equation}
F_{\beta,p}\colon \pi_1(M,p)\longrightarrow \R,\qquad \sigma\longmapsto \frac{1}{2\pi}\int_\sigma \frac{1}{i\beta}d\beta.
\end{equation}
Notice that, in the notation above,
\begin{equation}\label{eq:pullback}
\frac{1}{2\pi}\int_\sigma \frac{1}{i\beta}d\beta=\frac{1}{2\pi}\int_{\beta\circ\sigma} \frac{1}{ia}da=\frac{1}{2\pi}\int_{\beta\circ\sigma} \psi,
\end{equation}
so, $F$ has image in $\Z$.
\begin{lemma}\label{lemma:cohomological_condition_square}
Let $M$ be a smooth manifold and $\beta\colon M\to \U(1)$, then there exists $b\colon M\to\U(1)$ such that $\beta=b^2$ if and only if for all $p\in M$, the functional
\begin{align*}
[F_{\beta,p}]\colon \pi_1(M,p)\longrightarrow \Z_2,\qquad \sigma\longmapsto \frac{1}{2\pi}\int_\sigma \frac{1}{i\beta}d\beta\ \textrm{mod}\ 2
\end{align*}
is zero.
Explicitly, given $y_p\in\U(1)$ such that $y_p^2=\beta(p)$, then for all $q\in M$ in the same connected component of $p$,
\begin{equation}\label{eq:root_beta}
b(q)=y_p \exp\left(\frac{1}{2} \int_\sigma \frac{1}{\beta}d\beta\right)
\end{equation}
for all continuous $\sigma\colon [0,1]\to M$ such that $\sigma(0)=p$ and $\sigma(1)=q$.
\end{lemma}
\begin{proof}
If $\beta=b^2$ for some $b\colon M\to \U(1)$, then for all $p\in M$ and $\sigma\in\pi_1(M,p)$,
\begin{equation}
F_{\beta,p}(\sigma)=\frac{1}{2\pi}\int_\sigma \frac{2}{ib}db=2\left(\frac{1}{2\pi}\int_{b\circ\sigma} \psi\right).
\end{equation}
Since $b\circ \sigma$ is a loop, $F_{\beta,p}(\sigma)$ is even, so $[F_{\beta,p}]=0$.

Conversely, choose a point in every connected component of $M$ and define $b$ by glueing functions defined as in \eqref{eq:root_beta}.
We can verify $\beta=b^2$ on each connected component, so let $p$ be the chosen point in said component.
Connected components on manifolds are also path connected components, so for all $q$ in the same connected component, there exists a smooth $\sigma\colon[0,1]\to M$ such that $\sigma(0)=p$ and $\sigma(1)=q$.
The value $b(q)$ is independent from the path $\sigma$ chosen, in fact if we pick another such $\sigma'\colon [0,1]\to M$, then the composition of paths $(\sigma')^{-1} \ast \sigma\in\pi_1(M,p)$ is a loop, and thus
\begin{align*}
\int_\sigma \frac{1}{\beta}d\beta-\int_{\sigma'} \frac{1}{\beta}d\beta
&=\int_\sigma \frac{1}{\beta}d\beta+\int_{(\sigma')^{-1}} \frac{1}{\beta}d\beta
=2\pi i\bigg(\frac{1}{2\pi}\int_{(\sigma')^{-1}\ast \sigma} \frac{1}{i\beta}d\beta\bigg)
=4\pi i k
\end{align*}
for some $k\in\Z$.
It follows that
\begin{equation}
y_p \exp\left(\frac{1}{2} \int_\sigma \frac{1}{\beta}d\beta\right)
=y_p \exp\left(\frac{1}{2} \int_{\sigma'} \frac{1}{\beta}d\beta+2\pi i k\right)=
y_p \exp\left(\frac{1}{2} \int_{\sigma'} \frac{1}{\beta}d\beta\right)
\end{equation}
We can now compute
\begin{align*}
b^2(q)
&=y_p^2 \left(\exp\left(\frac{1}{2} \int_\sigma \frac{1}{\beta}d\beta\right)\right)^2
=\beta(p) \exp\left(\int_\sigma \frac{1}{\beta}d\beta\right).
\end{align*}
Since locally $\frac{1}{\beta}d\beta=d\log(\beta)$, up to picking a suitable partition of $[0,1]$ we have $\exp\left(\int_\sigma \frac{1}{\beta}d\beta\right)=\beta(q)/\beta(p)$ so $b^2(q)=\beta(q)$.
\end{proof}
We deduce
\begin{prop}
Let $M$ be a smooth manifold and $\beta\colon M\to\U(1)$, then the following are equivalent:
\begin{enumerate}
\item\label{prop:cohomological_char_root:1} there exists $b\colon M\to\U(1) $ such that $\beta=b^2$;
\item\label{prop:cohomological_char_root:2} $[F_{\beta,p}]=0$ for all $p\in M$;
\item\label{prop:cohomological_char_root:3} $[F_{\beta,p_k}](\sigma_k)=0$ for a set of loops $\sigma_k\in \pi_1(M,p_k)$ corresponding to a set of generators of $H_1(M,\Z)$;
\item\label{prop:cohomological_char_root:4} $[F_{\beta,p_k}](\sigma_k)=0$ for a set of loops $\sigma_k\in \pi_1(M,p_k)$ corresponding to a set of generators of $H_1(M,\Z_2)=H_1(M,\Z)\otimes_{\Z}\Z_2$;
\item\label{prop:cohomological_char_root:5} the pullback $\beta^*\colon \Z_2\cong H^1(\U(1),\Z_2)\to H^1(M,\Z_2)$ is the zero map.
\end{enumerate}
\end{prop}
\begin{proof}
The equivalence \ref{prop:cohomological_char_root:1}$\Leftrightarrow$\ref{prop:cohomological_char_root:2} is Lemma \ref{lemma:cohomological_condition_square}.

For \ref{prop:cohomological_char_root:2}$\Leftrightarrow$\ref{prop:cohomological_char_root:3}, suppose at first that $M$ is connected and let $p\in M$.
The functional $[F_{\beta,p}]\colon \pi_1(M,p)\to \Z_2$ is a group homomorphism and by Hurewicz theorem, $H_1(M,\Z)$ is the abelianisation of $\pi_1(M,p)$.
Since $\Z_2$ is an abelian group, there are natural isomorphisms
\begin{equation}
\Hom(\pi(M,p),\Z_2)\cong \Hom(H_1(M,\Z),\Z_2)= \Hom_{\Z}(H_1(M,\Z),\Z_2),
\end{equation}
and thus, there is a canonical factorization of $[F_{\beta,p}]$ as an abelian group homomorphism (i.e.\ $\Z$-linear map) $H_1(M,\Z)\to\Z_2$ which is the zero map if and only if $[F_{\beta,p}]$ is zero.
In particular this proves \ref{prop:cohomological_char_root:2}$\Leftrightarrow$\ref{prop:cohomological_char_root:3}.
In general, $M=\coprod_{i\in I} M_i$ with $M_i$ connected for all $i\in I$, so $H_1(M,\Z)=\bigoplus_{i\in I}H_1(M_i,\Z)$ and hence
\begin{equation}
\Hom_{\Z}(H_1(M,\Z),\Z_2)
\cong \prod_{i\in I}\Hom_{\Z}(H_1(M_i,\Z),\Z_2)
\cong \prod_{i\in I}\Hom(\pi_1(M_i,p_i),\Z_2).
\end{equation}
Thus, \ref{prop:cohomological_char_root:2} holds if and only if $[F_{\beta,p_i}]=0$ for all $i\in I$, and by the previous isomorphism, this happens if and only if \ref{prop:cohomological_char_root:3}.

\ref{prop:cohomological_char_root:3}$\Leftrightarrow$\ref{prop:cohomological_char_root:4} follows from properties of tensor products and linear maps, in fact, given a $\Z$-module $A$, a $\Z$-linear map $A\to\Z_2$ vanishes on $2A$, and thus factors as a map $A/2A\to\Z_2$.
Moreover, $A/2A\cong A\otimes_{\Z}\Z_2$ (seen by applying the right-exact functor $A\otimes_{\Z}\cdot$ to the short exact sequence $0\rightarrow \Z\xrightarrow{2\cdot} \Z\rightarrow\Z_2\rightarrow 0$).
From these properties along with the homological universal coefficients theorem, we find the following natural isomorphisms
\begin{equation}
\Hom_{\Z}(H_1(M,\Z),\Z_2)\cong \Hom_{\Z_2}(H_1(M,\Z)\otimes_{\Z} \Z_2,\Z_2)\cong\Hom_{\Z_2}(H_1(M,\Z_2),\Z_2),
\end{equation}
that show the equivalence \ref{prop:cohomological_char_root:3}$\Leftrightarrow$\ref{prop:cohomological_char_root:4}.

Finally, we prove \ref{prop:cohomological_char_root:3}$\Leftrightarrow$\ref{prop:cohomological_char_root:5} by the cohomological universal coefficient theorem, which gives the natural isomorphism $\Hom(H_1(M,\Z),\Z_2)\cong H^1(M,\Z_2)$.
In particular, the class in $H^1(M,\Z_2)$ corresponding to $[F_{\beta,p}]$, is by construction the pullback along $\beta$ of the fundamental form on $\U(1)$ (see \eqref{eq:pullback}).
Since $H^1(\U(1),\Z_2)$ is generated by the  integral functional associated to the fundamental form, this image is zero if and only if the whole $\beta^*$ is the zero map.
\end{proof}
This proposition clarifies when two structures built as in Proposition \ref{prop:PSKcerchioConnessioni} are isomorphic as in Proposition \ref{prop:isomorphism_if_root}.
Since $\beta^*$ always vanishes on torsion elements, $H^1_{dR}(M)=0$ is a sufficient condition for not only $\beta^*\colon H^1_{dR}(\U(1))\to H^1_{dR}(M)$ being zero, but also for $\beta^*\colon H^1(\U(1),\Z_2)\to H^1(M,\Z_2)$ being zero.
However, the condition $\beta^*=0$ on the cohomology with coefficients in $\Z_2$ is in general strictly weaker than the same condition in de Rham cohomology.
\section{Complex hyperbolic n-space}
In this section we are going to describe a special family of projective special K\"{a}hler manifolds, which can be thought of as the simplest possible model in a given dimension.

Let $\C^{n,1}$ be the Hermitian space $\C^{n+1}$ endowed with the Hermitian form
\begin{equation}
\langle z,w\rangle =\overline{z_1} w_1+\dots+\overline{z_n} w_n-\overline{z_{n+1}} w_{n+1}.
\end{equation}
It is a complex vector space, so it makes sense to consider the projective space associated to it, that is $\Pj(\C^{n,1})=(\C^{n,1}\setminus \{0\})/\C^*$ with the quotient topology and the canonical differentiable structure, where $\C^*$ acts by scalar multiplication.
We will denote the quotient class corresponding to an element $z\in\C^{n,1}$ by $[z]$.
We can define the following open subset:
\begin{equation}
\Hyp_{\C}^n:=\{[v]\in\Pj(\C^{n,1})|\langle v,v\rangle<0\}.
\end{equation}
Let $v=(v_1,\dots,v_{n+1})\in\C^{n,1}$, notice that if $[v]\in\Hyp_{\C}^n$, then $|v_1|^2+\dots+|v_n|^2-|v_{n+1}|^2<0$ so $|v_{n+1}|^2>|v_1|^2+\dots+|v_n|^2\ge 0$ which implies $v_{n+1}\neq 0$.
We thus have a global differentiable chart $\Hyp_{\C}^n\to\C^n$ by restricting the projective chart $[v]\mapsto \Big(\frac{v_1}{v_{n+1}},\dots,\frac{v_n}{v_{n+1}}\Big)$.
\begin{rmk}\label{rmk:HypSpContractible}
The inverse of this chart $\C^n\to\Pj(\C^{n,1})$ maps $z=(z_1,\dots,z_n)\in\C^n$ to $[(z_1,\dots,z_n,1)]$, which is in $\Hyp_{\C}^n$ if and only if $\norm{z}^2<1$.
We have proven that $\Hyp_{\C}^n$ is diffeomorphic to the complex unit ball and thus in particular it is contractible.
\end{rmk}

Consider now the Lie group $\SU(n,1)$ of the matrices with determinant $1$ that are unitary with respect to the Hermitian metric on $\C^{n,1}$.
We define a left action of $\SU(n,1)$ on $\Hyp_{\C}$ such that $A[v]=[Av]$; it is well defined by linearity and invertibility and it is smooth.

This action is also transitive, in fact given $[v],[w]\in\Hyp_{\C}^n$, without loss of generality, we can assume that $\langle v,v\rangle=-1=\langle w,w\rangle$.
Because of this, we can always complete $v$ and $w$ to an orthonormal basis with respect to the Hermitian product, obtaining $\{v_1,\dots,v_n,v\}$ and $\{w_1,\dots,w_n,w\}$.
Consider the following block matrices $V=(v_1|\dots|v_n|v)$ and $W=(w_1|\dots|w_n|w)$ which, up to permuting two of the first $n$-columns, belong to $\SU(n,1)$.
The matrix $A=WV^{-1}\in\SU(n,1)$ maps $v$ in $w$ and thus $[v]$ in $[w]$.

We shall now compute the stabiliser of the last element of the canonical basis $e_{n+1}$ for this action, that is, the set of matrices $A\in\SU(n,1)$ such that $Ae_{n+1}=\lambda e_{n+1}$ for $\lambda\in\C$.
Observe that $\lambda\in \U(1)$ since
\begin{equation}
-1=\langle e_{n+1}, e_{n+1}\rangle=\langle Ae_{n+1}, Ae_{n+1}\rangle=\langle \lambda e_{n+1}, \lambda e_{n+1}\rangle=-|\lambda|^2.
\end{equation}
Moreover, the last column of $A$ is $A_{n+1}=Ae_{n+1}=\lambda e_{n+1}$.
This forces $A$ to assume the form
\begin{equation}
\begin{pmatrix}
B&0\\
0&\lambda
\end{pmatrix}.
\end{equation}
Since $A$ belongs to $\SU(n,1)$, we must infer that $B$ belongs to $\U(n)$ and $\lambda=\det(B)^{-1}$.
The stabiliser of $e_{n+1}$ is thus $\spc(\U(n)\U(1))$, which is isomorphic to $\U(n)$.
We deduce that $\Hyp_{\C}^n$ is a symmetric space $\SU(n,1)/\spc(\U(n)\U(1))$.

We will adopt the nomenclature of \cite{Goldman} for the following
\begin{defi}
We call the K\"{a}hler manifold $\Hyp_{\C}^n$ of complex dimension $n$ the \emph{complex hyperbolic $n$-space}.
\end{defi}

There is a natural K\"{a}hler structure on $\Hyp_{\C}^n$ coming from its representation as a symmetric space $G/H$.

Let $\lie{g}=\lie{h}+\lie{m}$ be the canonical decomposition, in particular
\begin{equation}
\lie{m}:=\bigg\{\begin{pmatrix}
0&x\\
x^\star&0
\end{pmatrix}\big| x\in\C^n\bigg\}.
\end{equation}

On a symmetric space, there is a one-to-one correspondence between Riemannian metrics and $\Ad(H)$-invariant positive definite symmetric bilinear forms on $\lie{m}$ (See \cite[II, Corollary 3.2, p.\ 200]{KN}).
Let $\theta\colon T_{[e_{n+1}]}\Hyp_{\C}^n\cong \lie{m}\to\C^n$ be the identification mapping to $x$ the tangent vector corresponding to $\begin{pmatrix}
0&x\\
x^\star&0
\end{pmatrix}$.
With this identification, for $A\in\U(n)$ we see that the $\Ad(A)$-action on $\lie{m}$ corresponds on $\C^n$ to the $x\mapsto\det(A)Ax$.

The metric is induced by the Killing form on $\lie{su}(n,1)$ given by (\cite{Helgason})
\begin{equation}
B(X,Y)=2(n+1)\tr(XY),\quad \forall X,Y\in\lie{u}(n,1).
\end{equation}
We restrict the Killing form to $\lie{m}$ in order to define an $\Ad(H)$-invariant bilinear form, that is, given $x,y\in \C^n$, if $X,Y$ are the corresponding tangent vectors,
\begin{align*}
B(X,Y)
&=2(n+1)\tr\left(\begin{pmatrix}
0&x\\
x^\star&0
\end{pmatrix}\begin{pmatrix}
0&y\\
y^\star&0
\end{pmatrix}\right)
=2(n+1)\tr\begin{pmatrix}
x y^\star&0\\
0&x^\star y
\end{pmatrix}\\
&=2(n+1)\Re(x^\star y)
=2(n+1)(\theta^\star \theta) (X,Y).
\end{align*}
We define $g_{[e_{n+1}]}:=\theta^\star \theta$, which is $\Ad(\U(n))$-invariant, so it extends to a global Riemannian metric $g$.
By using the same idea, we can also define an almost complex structure $I$ on $\lie{m}$ as the map corresponding to the scalar multiplication by $i$ on $\C^n$.
This structure is compatible with the metric and it is $\Ad(\U(n))$-invariant, so it defines a K\"{a}hler structure (see \cite[II, Proposition 9.3, p.\ 260]{KN}).
The K\"{a}hler form $\omega$ is then:
\begin{align*}
\omega(X,Y)
=g(IX,Y)
=\Re(x^\star i^\star y)
=\Im(x^\star y)
=\Im(\theta^\star \otimes \theta)(X,Y).
\end{align*}

\begin{prop}
The manifold $\Hyp_{\C}^n$ has curvature tensor $-\Omega_{\Pj_{\C}^n}$ and is projective special K\"{a}hler for all $n\ge 1$ with constant zero deviance.
\end{prop}
\begin{proof}
The computation of the curvature tensor is standard.
By Remark \ref{rmk:HypSpContractible}, we know that $\Hyp_{\C}^n$ is contractible, allowing us to apply Corollary \ref{cor:charPSKesatte}.
If we choose as tensor $\eta$ of type $\sharp_2 S_{3,0}M$ the $0$-section, then the differential condition \ref{eq:condDiffCor} is trivially satisfied, while condition \ref{eq:condCurvCor} follows from the computation of the curvature tensor.
\end{proof}

Notice that the deviance measures the difference of a projective special K\"{a}hler manifold of dimension $2n$ from being the complex hyperbolic $n$-space.
More precisely, we have
\begin{prop}
At a point $p$ of a projective special K\"{a}hler manifold $M$ with intrinsic deviance $\gamma\colon S\to \sharp_2 S_{3,0}M$, the curvature tensor $\Omega_M$ coincides with the one of $\Hyp_{\C}^n$ exactly in those points $p$ where $\gamma|_p$ vanishes.
\end{prop}
In particular, for any section of $S$ defined on an open neighbourhood of $p$, the corresponding local deviance vanishes at $p$ whenever the two curvatures coincide.
\begin{proof}
One direction follows from condition \ref{theo:charpsk:CurvCond}.
For the opposite one, if $\Omega_M=\Omega_{\Hyp_{\C}^{n}}=-\Omega_{\Pj_{\C}^n}$, then $\scal_M=-2(n+1)$ and the intrinsic deviance vanishes as the norm of any local deviance vanishes by \eqref{eq:ScalCurv}.
\end{proof}
We can also prove
\begin{prop}\label{prop:uniquenessHyp}
The only complete connected and simply connected projective special K\"{a}hler manifold of dimension $2n$ with zero deviance is $\Hyp_{\C}^n$.
\end{prop}
\begin{proof}
Let $(\pi\colon \widetilde{M}\to M,\nabla)$ be such a projective special K\"{a}hler manifold.
Consider a point $p\in M$, then $(T_p M,g,I)$ can be seen as a complex vector space compatible with the metric and can thus be identified with the tangent space at a point of $\Hyp_\C^n$ via an isomorphism $F$ as they are both isomorphic to $\C^n$ with the standard metric.
Being complex manifolds, $\Hyp_\C^n$ and $M$ are analytic, and since the curvature of $M$ is forced to be $-\Omega_{\Pj^n_{\C}}$, which corresponds to a $\lie{u}(n)$-invariant map from the bundle of unitary frames to $S^2(\lie{u}(n))$, it is also parallel with respect to the Levi-Civita connection.
It follows that the linear isomorphism $F$ preserves the curvature tensors and their covariant derivatives.
Thus, $F$ can be extended to a diffeomorphism $f\colon M\to\Hyp_\C^n$ (See \cite[I, Corollary 7.3, p.\ 261]{KN}) such that $F$ is its differential at $p$.

Since $F$ preserves $I$ and $\omega$ which are parallel, $f$ is an isomorphism of K\"{a}hler manifolds, as the latter maps parallel tensors to parallel tensors.
Since the deviance of both manifolds is zero, we also have an isomorphism of projective special K\"{a}hler manifolds.
\end{proof}
\section{Classification of projective special K\"{a}hler Lie groups in dimension 4}
If $M$ is a Lie group, the conditions of Theorem \ref{theo:characterisationPSK} are simpler, because a Lie group is always parallelisable.
As a consequence, the bundle $\sharp_2 S_{3,0}(M)$ is trivial, and in particular we have a global coordinate system to write the local deviances.

\begin{defi}\label{def:PSKLg}
A projective special K\"{a}hler Lie group is a Lie group with projective special K\"{a}hler structure such that the K\"{a}hler structure is left-invariant.
\end{defi}
Notice that we do not require the deviance to be left-invariant.

An example is $\Hyp^n_{\C}$, since the Iwasawa decomposition $\SU(n,1)=KAN$ (see \cite[Theorem 1.3, p.\ 403]{Helgason}) gives a left-invariant K\"{a}hler structure on the solvable Lie group $AN$.
We denote by $\Hyp_\lambda$ the hyperbolic plane with curvature $-\lambda^2$, which is actually just a rescaling of $\Hyp_{\C}^1$.

With Definition \ref{def:PSKLg}, we are able to classify $4$-dimensional projective special K\"{a}hler Lie groups; we obtain exactly two, which coincide with the two $4$-dimensional cases appearing in the classification of projective special K\"{a}hler manifolds homogeneous under the action of a semisimple Lie groups (\cite{CortesClassHomoSS}).
\begin{theo}\label{theo:classificazionePSK4}
Up to isomorphisms of projective special K\"{a}hler manifolds, there are only two $4$-dimensional connected and simply connected projective special K\"{a}hler Lie groups: $\Hyp_{\sqrt{2}}\times\Hyp_{2}$ and the complex hyperbolic plane.

Up to isomorphisms that also preserve the Lie group structure, there are four families of $4$-dimensional connected and simply connected projective special K\"{a}hler Lie groups, listed in Table \ref{table:PSKLieGroups}.
\begin{table}[!ht]
\centering
\begin{tabular}[c]{|c|l|c|c|}
\hline
PSK &Diff.\ complex unitary cof.\ $\theta$ &Riemann curv.\ &$\sigma$\\
\hline
$\Hyp_{\sqrt{2}}\times\Hyp_2$&
\begin{tabular}{@{}l@{}}
$d\theta^1=-\frac{\sqrt{2}}{2}\overline{\theta^1}\wedge\theta^1$\\
$d\theta^2=-\overline{\theta^2}\wedge\theta^2$\\
\end{tabular}
&
\begin{tabular}{@{}l@{}}
$\frac{1}{2}(\overline{\theta^1}\wedge\theta^1)^2$\\
$+(\overline{\theta^2}\wedge\theta^2)^2$
\end{tabular}
&
$\frac{3}{2}(\theta^1)^2\theta^2$\\
\hline
$\Hyp_{\C}^2$&
\begin{tabular}{@{}l@{}}
$d\theta^1=\frac{1}{2}\theta^1\wedge(\theta^2+\overline{\theta^2})$\\
$d\theta^2=-\overline{\theta^1}\wedge\theta^1-\overline{\theta^2}\wedge\theta^2$\\
\end{tabular}
&
$-\Omega_{\Pj^{2}_{\C}}^{\flat}$
&0\\
\hline
$\Hyp_{\C}^2$&
\begin{tabular}{@{}l@{}}
$d\theta^1=(\frac{1}{2}-\frac{i}{\delta})\theta^1\wedge(\theta^{2}+\overline{\theta^2})$\\
$d\theta^2=-\overline{\theta^1}\wedge\theta^1-\overline{\theta^2}\wedge\theta^2$\\
$\delta>0$\\
\end{tabular}
&
$-\frac{1}{\delta}\Omega_{\Pj_{\C}^2}^{\flat}$
&0\\
\hline
$\Hyp_{\C}^2$&
\begin{tabular}{@{}l@{}}
$d\theta^1=-(\frac{1}{\delta}+\frac{i}{2})\theta^1\wedge(\theta^{2}-\overline{\theta^2})$\\
$d\theta^2=-i\overline{\theta^1}\wedge\theta^1-\overline{\theta^2}\wedge\theta^2$\\
$\delta>0$\\
\end{tabular}
&
$-\frac{1}{\delta}\Omega_{\Pj^{2}_{\C}}^{\flat}$
&0\\
\hline
\end{tabular}
\caption{Connected projective special K\"ahler Lie groups.}
\label{table:PSKLieGroups}
\end{table}
\end{theo}
\begin{proof}
We will study K\"ahler Lie groups through K\"ahler Lie algebras, see e.g.\ \cite[\S 1.1, p.\ 26]{DorfmeisterNakajima}.
We will start from the classification of pseudo-K\"{a}hler Lie groups provided by \cite{Ovando2004}.
Table \ref{table:classificationOvando} displays the eighteen families of non-abelian pseudo-K\"{a}hler Lie algebras in dimension $4$.
\begin{table}[!ht]
\centering
\begin{tabular}[c]{|l|l|p{6cm}|}
\hline
$\lie{g}$		&		$I$		&		$\omega$\\
\hline
$\lie{rh}_3$	&$Ie_1=e_2, Ie_3=e_4$&$a_1 (e^{13}+e^{24})+a_2(e^{14}-e^{23})+a_3 e^{12}$, $a_1^2+a_2^2\neq 0$\\
\hline
$\lie{rr}_{3,0}$	&$Ie_1=e_2, Ie_3=e_4$&$a_1 e^{12}+a_2 e^{34}$, $a_1a_2\neq 0$\\
\hline
$\lie{rr}'_{3,0}$	&$Ie_1=e_4, Ie_2=e_3$&$a_1 e^{14}+a_2 e^{23}$, $a_1a_2\neq 0$\\
\hline
$\lie{r}_2\lie{r}_2$&$Ie_1=e_2, Ie_3=e_4$&$a_1 e^{12}+a_2 e^{34}$, $a_1a_2\neq 0$\\
\hline
$\lie{r}'_2$	&$Ie_1=e_3, Ie_2=e_4$&$a_1 (e^{13}-e^{24})+a_2(e^{14}+e^{23})$, $a_1^2+a_2^2\neq 0$\\
\hline
$\lie{r}'_2$	&$Ie_1=-e_2, Ie_3=e_4$&$a_1 (e^{13}-e^{24})+a_2(e^{14}+e^{23})+a_3 e^{12}$, $a_1^2+a_2^2\neq 0$\\
\hline
$\lie{r}_{4,-1,-1}$	&$Ie_4=e_1, Ie_2=e_3$&$a_1 (e^{12}+e^{34})+a_2(e^{13}-e^{24})+a_3 e^{14}$, $a_1^2+a_2^2\neq 0$\\
\hline
$\lie{r}'_{4,0,\delta}$	&$Ie_4=e_1, Ie_2=e_3$&$a_1 e^{14}+a_2 e^{23}$, $a_1a_2\neq 0$, $\delta>0$\\
\hline
$\lie{r}'_{4,0,\delta}$	&$Ie_4=e_1, Ie_2=-e_3$&$a_1 e^{14}+a_2 e^{23}$, $a_1a_2\neq 0$, $\delta>0$\\
\hline
$\lie{d}_{4,1}$	&$Ie_1=e_4, Ie_2=e_3$&$a_1 (e^{12}-e^{34})+a_2e^{14}$, $a_1\neq 0$\\
\hline
$\lie{d}_{4,2}$	&$Ie_4=-e_2, Ie_1=e_3$&$a_1 (e^{14}+e^{23})+a_2e^{24}$, $a_1\neq 0$\\
\hline
$\lie{d}_{4,2}$	&$Ie_4=-2e_1, Ie_2=e_3$&$a_1 e^{14}+a_2 e^{23}$, $a_1a_2\neq 0$\\
\hline
$\lie{d}_{4,1/2}$	&$Ie_4=e_3, Ie_1=e_2$&$a_1 (e^{12}-e^{34})$, $a_1\neq 0$\\
\hline
$\lie{d}_{4,1/2}$	&$Ie_4=e_3, Ie_1=-e_2$&$a_1 (e^{12}-e^{34})$, $a_1\neq 0$\\
\hline
$\lie{d}'_{4,\delta}$	&$Ie_4=e_3, Ie_1=e_2$&$a_1 (e^{12}-\delta e^{34})$, $a_1\neq 0$, $\delta>0$\\
\hline
$\lie{d}'_{4,\delta}$	&$Ie_4=-e_3, Ie_1=e_2$&$a_1 (e^{12}-\delta e^{34})$, $a_1\neq 0$, $\delta>0$\\
\hline
$\lie{d}'_{4,\delta}$	&$Ie_4=-e_3, Ie_1=-e_2$&$a_1 (e^{12}-\delta e^{34})$, $a_1\neq 0$, $\delta>0$\\
\hline
$\lie{d}'_{4,\delta}$	&$Ie_4=e_3, Ie_1=-e_2$&$a_1 (e^{12}-\delta e^{34})$, $a_1\neq 0$, $\delta>0$\\
\hline
\end{tabular}
\caption[Pseudo-K\"{a}hler Lie algebras]{Classification of $4$-dimensional pseudo-K\"{a}hler non-abelian Lie algebras \cite[Table 5.1, p.\ 63]{Ovando2004}.}
\label{table:classificationOvando}
\end{table}

Among these families, only for the ones in Table \ref{table:KahlerLieAlgebras} the metric can be positive definite, i.e.\ K\"{a}hler.
\begin{table}[!ht]
\centering
\begin{tabular}[c]{|l|l|l|l|l|}
\hline
Case& $\lie{g}$		&		$I$		&		$\omega$&	Conditions\\
\hline
I&$\lie{rr}_{3,0}$	&$Ie_1=e_2, Ie_3=e_4$&$a_1 e^{12}+a_2 e^{34}$& $a_1, a_2>0$\\
\hline
II&$\lie{rr}'_{3,0}$	&$Ie_1=e_4, Ie_2=e_3$&$a_1 e^{14}+a_2 e^{23}$& $a_1, a_2> 0$\\
\hline
III&$\lie{r}_2\lie{r}_2$&$Ie_1=e_2, Ie_3=e_4$&$a_1 e^{12}+a_2 e^{34}$& $a_1, a_2> 0$\\
\hline
IV&$\lie{r}'_{4,0,\delta}$	&$Ie_4=e_1, Ie_2=e_3$&$a_1 e^{14}+a_2 e^{23}$& $a_1<0;a_2,\delta>0$\\
\hline
V&$\lie{r}'_{4,0,\delta}$	&$Ie_4=e_1, Ie_2=-e_3$&$a_1 e^{14}+a_2 e^{23}$& $a_1,a_2<0; \delta>0$\\
\hline
VI&$\lie{d}_{4,2}$	&$Ie_4=-2e_1, Ie_2=e_3$&$a_1 e^{14}+a_2 e^{23}$& $a_1,a_2> 0$\\
\hline
VII&$\lie{d}_{4,1/2}$	&$Ie_4=e_3, Ie_1=e_2$&$a_1 (e^{12}-e^{34})$& $a_1> 0$\\
\hline
VIII&$\lie{d}'_{4,\delta}$	&$Ie_4=e_3, Ie_1=e_2$&$a_1 (e^{12}-\delta e^{34})$& $a_1,\delta>0$\\
\hline
IX&$\lie{d}'_{4,\delta}$	&$Ie_4=-e_3, Ie_1=-e_2$&$a_1 (e^{12}-\delta e^{34})$& $a_1<0;\delta>0$\\
\hline
\end{tabular}
\caption{Non-abelian K\"{a}hler Lie algebras of dimension $4$.}
\label{table:KahlerLieAlgebras}
\end{table}
It is now straightforward to find a unitary frame $u$ for each case, that is such that $g=\sum_{k=1}^4 (u^k)^2$, $Iu_1=u_2$, $Iu_3=u_4$ and $\omega=u^{1,2}+u^{3,4}$.
With respect to $u$, we can write the new structure constants and compute the Levi-Civita connection form $\omega^{LC}$ and the corresponding curvature form $\Omega^{LC}$.
We write
\begin{align*}
H_1:=\begin{pmatrix}
&\quad -u^{12}&\rvline&&\\
u^{12}&&\rvline&&\\
\hline
&&\rvline&&\\
&&\rvline&&
\end{pmatrix},& &
H_2=\begin{pmatrix}
&&\rvline&&\\
&&\rvline&&\\
\hline
&&\rvline&&-u^{34}\\
&&\rvline&u^{34}&
\end{pmatrix}.
\end{align*}
\begin{table}
\centering
\begin{tabular}[c]{|c|c|l|c|}
\hline
Case& $\lie{g}$ &Str.\ constants & $\Omega^{LC}$\\
\hline
I&$\lie{rr}_{3,0}$	&
\begin{tabular}{@{}l@{}}
$[u_1,u_2]=a u_2$\\
$a>0$\\
\end{tabular}
&
$a^2 H_1$
\\
\hline
II&$\lie{rr}'_{3,0}$	&
\begin{tabular}{@{}l@{}}
$[u_1,u_3]=- u_4$\\
$[u_1,u_4]= u_3$\\
\end{tabular}
&
0
\\
\hline
III&$\lie{r}_2\lie{r}_2$&
\begin{tabular}{@{}l@{}}
$[u_1,u_2]=a u_2$\\
$[u_3,u_4]=b u_4$\\
$a,b>0$\\
\end{tabular}
&
$a^2 H_1+b^2 H_2$
\\
\hline
IV&$\lie{r}'_{4,0,\delta}$	&
\begin{tabular}{@{}l@{}}
$[u_1,u_2]=a u_2$\\
$[u_1,u_3]=-\delta a u_4$\\
$[u_1,u_4]=\delta a u_3$\\
$a,\delta>0$\\
\end{tabular}
&
$a^2H_1$
\\
\hline
V&$\lie{r}'_{4,0,\delta}$	&
\begin{tabular}{@{}l@{}}
$[u_1,u_2]=a u_2$\\
$[u_1,u_3]=\delta a u_4$\\
$[u_1,u_4]=-\delta a u_3$\\
$a,\delta>0$\\
\end{tabular}
&
$a^2 H_1$
\\
\hline
VI&$\lie{d}_{4,2}$	&
\begin{tabular}{@{}l@{}}
$[u_1,u_2]=-2a u_1$\\
$[u_1,u_3]=2a u_4$\\
$[u_2,u_3]=-a u_3$\\
$[u_2,u_4]=a u_4$\\
$a>0$\\
\end{tabular}
&
$-a^2\Omega_{\Pj^2_{\C}}-6a^2 H_2$
\\
\hline
VII&$\lie{d}_{4,1/2}$	&
\begin{tabular}{@{}l@{}}
$[u_1,u_2]=2a u_4$\\
$[u_1,u_3]=-a u_1$\\
$[u_2,u_3]=-a u_2$\\
$[u_3,u_4]=2a u_4$\\
$a>0$\\
\end{tabular}
&
$-a^2\Omega_{\Pj^2_{\C}}$
\\
\hline
VIII&$\lie{d}'_{4,\delta}$	&
\begin{tabular}{@{}l@{}}
$[u_1,u_2]=2a\sqrt{\delta} u_4$\\
$[u_1,u_3]=-a\sqrt{\delta} u_1+\frac{2a}{\sqrt{\delta}} u_2$\\
$[u_2,u_3]=-\frac{2a}{\sqrt{\delta}} u_1-a\sqrt{\delta} u_2$\\
$[u_3,u_4]=2a\sqrt{\delta} u_4$\\
$a,\delta>0$\\
\end{tabular}
&
$-\delta a^2\Omega_{\Pj^2_{\C}}$
\\
\hline
IX&$\lie{d}'_{4,\delta}$	&
\begin{tabular}{@{}l@{}}
$[u_1,u_2]=-2a\sqrt{\delta} u_3$\\
$[u_1,u_4]=-a\sqrt{\delta} u_1-\frac{2a}{\sqrt{\delta}} u_2$\\
$[u_2,u_4]=+\frac{2a}{\sqrt{\delta}} u_1-a\sqrt{\delta} u_2$\\
$[u_3,u_4]=-2a\sqrt{\delta} u_3$\\
$a,\delta>0$\\
\end{tabular}
&
$-\delta a^2 \Omega_{\Pj^2_{\C}}$
\\
\hline
\end{tabular}
\caption{Curvature tensors.}
\label{table:curvatures}
\end{table}

From Table \ref{table:curvatures} we notice that the curvature tensors are of two types:
\begin{enumerate}[label=(\roman*)]
\item\label{case:curvi} $a^2 H_1+b^2 H_2$ for $a,b\ge 0$;
\item\label{case:curvii} $-a^2(\Omega_{\Pj^2_{\C}}+6bH_2)$ for $a>0$ and $b\in\{0,1\}$.
\end{enumerate}

A K\"{a}hler Lie group $M$ of dimension $4$ is always solvable \cite[Theorem 9, p.\ 155]{Chu}.
This implies that $M$ is the product of a torus\index{torus} (product of circumferences) and a euclidean space \cite[Theorem 2\textsuperscript{a}, .\ 675]{Chevalley}, but $M$ is also simply connected, so it must be an euclidean space, and thus contractible.
If now $M$ has a projective special K\"{a}hler structure, thanks to Corollary \ref{cor:charPSKesatte}, there is a global section $\eta\colon M\to \sharp_2 S_{3,0}M$ satisfying \ref{eq:condCurvCor} and \ref{eq:condDiffCor}.
Applying $\flat_2$ we obtain a global section $\sigma$ of $S_{3,0}M$ which better displays the symmetry.

Consider the globally defined complex coframe $\theta^1=u^1+iu^2$, $\theta^2=u^3+i u^4$.
We write $\sigma$ in its generic form with respect to $\theta$:
\begin{equation}
\sigma=c_1(\theta^1)^3+c_2(\theta^1)^2\theta^2+c_3\theta^1(\theta^2)^2+c_4(\theta^2)^3
\end{equation}
for some functions $c_1,c_2,c_3,c_4\in\smooth{M,\C}$.
By raising the second index, we obtain $\eta=\sharp_2\sigma$ which is
\begin{align}
\eta
&=2c_1\theta^1\otimes \overline{\theta_1}\otimes \theta^1
+\frac{2c_2}{3}\left(\theta^1\otimes \overline{\theta_1}\otimes\theta^2+\theta^1\otimes \overline{\theta_2}\otimes\theta^1+\theta^2\otimes \overline{\theta_1}\otimes\theta^1\right)\\*
&\quad +\frac{2c_3}{3}\left(\theta^1\otimes \overline{\theta_2}\otimes\theta^2+\theta^2\otimes \overline{\theta_1}\otimes\theta^2+\theta^2\otimes \overline{\theta_2}\otimes\theta^1\right)
+2c_4\theta^2\otimes \overline{\theta_2}\otimes \theta^2.
\end{align}

With respect to this generic section, we can compute $[\eta\wedge\overline{\eta}]$ explicitly:
\allowdisplaybreaks[0]
\begin{align*}
[\eta\wedge\overline{\eta}]
=\frac{4}{9}\RE&\left(\overline{\theta^1}\wedge\theta^1\otimes \begin{pmatrix}
9|c_1|^2+|c_2|^2&3\overline{c_1}c_2+\overline{c_2}c_3\\
3\overline{c_2}c_1+\overline{c_3}c_2&|c_2|^2+|c_3|^2\\
\end{pmatrix}\right.\\
&\quad +\overline{\theta^1}\wedge\theta^2\otimes \begin{pmatrix}
3\overline{c_1}c_2+\overline{c_2}c_3&\overline{c_1}c_3+\overline{c_2}c_4\\
|c_2|^2+|c_3|^2&\overline{c_2}c_3+3\overline{c_3}c_4\\
\end{pmatrix}\\
&\quad +\overline{\theta^2}\wedge\theta^1\otimes \begin{pmatrix}
3\overline{c_2}c_1+\overline{c_3}c_2&|c_2|^2+|c_3|^2\\
\overline{c_3}c_1+\overline{c_4}c_2&\overline{c_3}c_2+3\overline{c_4}c_3\\
\end{pmatrix}\\
&\quad +\left.\overline{\theta^2}\wedge\theta^2\otimes \begin{pmatrix}
|c_2|^2+|c_3|^2&\overline{c_2}c_3+\overline{c_3}c_4\\
\overline{c_3}c_2+3\overline{c_4}c_3&|c_3|^2+9|c_4|^2\\
\end{pmatrix}\right).
\end{align*}
\allowdisplaybreaks
Notice that if we define $v_1,v_2,v_3\in \smooth{M,\C^2}$ such that
\begin{align}\label{eq:defVk}
v_1:=\begin{pmatrix}
2c_1\\
\frac{2c_2}{3}
\end{pmatrix}=\begin{pmatrix}
x\\
y
\end{pmatrix},\quad v_2:=\begin{pmatrix}
\frac{2c_2}{3}\\
\frac{2c_3}{3}
\end{pmatrix}=\begin{pmatrix}
y\\
z
\end{pmatrix},\quad v_3:=\begin{pmatrix}
\frac{2c_3}{3}\\
2c_4
\end{pmatrix}=\begin{pmatrix}
z\\
w
\end{pmatrix},
\end{align}
then we have
\begin{align*}
[\eta\wedge\overline{\eta}]
&=\RE\left(\overline{\theta^1}\wedge\theta^1\otimes \begin{pmatrix}
\norm{v_1}^2&\langle v_1,v_2\rangle\\
\overline{\langle v_1,v_2\rangle}&\norm{v_2}^2\\
\end{pmatrix}
+\overline{\theta^1}\wedge\theta^2\otimes \begin{pmatrix}
\langle v_1,v_2\rangle&\langle v_1,v_3\rangle\\
\norm{v_2}^2&\langle v_2,v_3\rangle\\
\end{pmatrix}\right.\\
&\quad\left.+\overline{\theta^2}\wedge\theta^1\otimes \begin{pmatrix}
\overline{\langle v_1,v_2\rangle}&\norm{v_2}^2\\
\overline{\langle v_2,v_3\rangle}&\overline{\langle v_2,v_3\rangle}\\
\end{pmatrix}
+\overline{\theta^2}\wedge\theta^2\otimes \begin{pmatrix}
\norm{v_2}^2&\langle v_2,v_3\rangle\\
\overline{\langle v_2,v_3\rangle}&\norm{v_3}^2\\
\end{pmatrix}\right).
\end{align*}
In other words, the coefficients of $[\eta\wedge\overline{\eta}]$ are the pairwise Hermitian products of $v_1,v_2,v_3$.

Returning to the classification, if we write $H_1,H_2,\Omega_{\Pj^2_{\C}}$ with respect to the complex coframe, we notice that the positions corresponding to the mixed Hermitian products are always zero.
\begin{align*}
H_1&=\RE\left(\overline{\theta^1}\wedge\theta^1\otimes\begin{pmatrix}
\frac{1}{2}&0\\
0&0
\end{pmatrix}\right),\qquad
H_2=\RE\left(\overline{\theta^2}\wedge\theta^2\otimes\begin{pmatrix}
0&0\\
0&\frac{1}{2}
\end{pmatrix}\right),\\
\Omega_{\Pj^2_{\C}}&=\RE\left(\overline{\theta^1}\wedge\theta^1\otimes\begin{pmatrix}
-2&0\\
0&-1\\
\end{pmatrix}
+\overline{\theta^1}\wedge\theta^2\otimes\begin{pmatrix}
0&0\\
-1&0
\end{pmatrix}\right.\\
&\quad \left.+\overline{\theta^2}\wedge\theta^1\otimes\begin{pmatrix}
0&-1\\
0&0\\
\end{pmatrix}+\overline{\theta^2}\wedge\theta^2\otimes\begin{pmatrix}
-1&0\\
0&-2\\
\end{pmatrix}\right).
\end{align*}
As a consequence, for all cases, if \ref{eq:condCurvCor} holds, then $v_1,v_2,v_3$ must be orthogonal.

Now we will treat each case of possible curvature tensor separately.
\begin{enumerate}[label=(\roman*)]
\item Let $a,b\ge 0$ and $\Omega^{LC}=a^2 H_1+b^2 H_2$, then
\begin{equation}
\Omega^{LC}=\RE\left(\overline{\theta^1}\wedge\theta^1\otimes\begin{pmatrix}
\frac{a^2}{2}&0\\
0&0\\
\end{pmatrix}+\overline{\theta^2}\wedge\theta^2\otimes\begin{pmatrix}
0&0\\
0&\frac{b^2}{2}\\
\end{pmatrix}\right).
\end{equation}
So, by \ref{eq:condCurvCor}, $[\eta\wedge\overline{\eta}]=-\Omega^{LC}-\Omega_{\Pj^2_{\C}}$, which implies
\begin{align*}
\norm{v_1}^2=2-\frac{a^2}{2},&&\norm{v_2}^2=1,&&\norm{v_3}^2=2-\frac{b^2}{2}.
\end{align*}
These equalities translate to a linear system in the squared norms of $x,y,z,w$ introduced in \eqref{eq:defVk}, namely
\begin{equation}
\begin{cases}
|x|^2+|y|^2=2-\frac{a^2}{2}\\
|y|^2+|z|^2=1\\
|z|^2+|w|^2=2-\frac{b^2}{2}
\end{cases}.
\end{equation}
Its solutions are
\begin{align}\label{eq:systemHxHnorms}
\begin{cases}
|x|^2=1-\frac{a^2}{2}+s\\
|y|^2=1-s\\
|z|^2=s\\
|w|^2=2-\frac{b^2}{2}-s
\end{cases}&& \textrm{ for } s\in [0,1].
\end{align}
Imposing the orthogonality conditions $\langle v_1,v_2\rangle=\langle v_2,v_3\rangle=\langle v_3,v_4\rangle=0$, we get:
\begin{align}\label{eq:systemHxHprod}
\begin{cases}
\overline{x}y+\overline{y}z=0\\
\overline{y}z+\overline{z}w=0\\
\overline{x}z+\overline{y}w=0\\
\end{cases}.
\end{align}
Notice that because of \eqref{eq:systemHxHnorms}, $y$ and $z$ cannot vanish simultaneously, so we have (at each point) three different cases:
\begin{itemize}
\item
Suppose at first that $z=0$, then $s=0$ and $\norm{y}=1$, so $y\neq 0$ and \eqref{eq:systemHxHprod} becomes
\begin{align}
\begin{cases}
\overline{x}y=0\\
0=0\\
\overline{y}w=0\\
\end{cases}.
\end{align}
Implying $x=w=0$, so the solutions are $(x,y,z,w)=(0,y,0,0)$ for $y\in \smooth{U,\U(1)}$.
Now $M$ is simply connected, so $y=e^{i\alpha}$ for some $\alpha\in\smooth{M}$, as $y$ lifts to the universal cover $\exp\colon i\R\to \U(1)$.
Thus we have $(c_1,c_2,c_3,c_4)=(0,\frac{3}{2}e^{i\alpha},0,0)$ for some $\alpha\in\smooth{M}$.
Finally, \eqref{eq:systemHxHnorms} gives
\begin{equation}
\begin{cases}
1-\frac{a^2}{2}=0\\
2-\frac{b^2}{2}=0
\end{cases}
\end{equation}
and thus $a=\sqrt{2}$ and $b=2$.
\item Suppose now that $z\neq 0$ and $y=0$, then \eqref{eq:systemHxHprod} becomes
\begin{align}
\begin{cases}
0=0\\
\overline{z}w=0\\
\overline{x}z=0\\
\end{cases}
\end{align}
and then $w=x=0$ so, similarly to the previous case, the solutions are $(c_1,c_2,c_3,c_4)=(0,0,e^{i\alpha},0)$ for $\alpha\in\smooth{M}$ and this time, \eqref{eq:systemHxHnorms} implies $a=2$ and $b=\sqrt{2}$.
\item The remaining case has $z\neq 0$ and $y\neq 0$.
In order to solve it, let us call $t:=\overline{y}z\neq 0$, then \eqref{eq:systemHxHnorms} and \eqref{eq:systemHxHprod} give
\begin{align*}
z&=\frac{ty}{|y|^2}=\frac{ty}{1-s},\\
x&=-\frac{\overline{t}y}{|y|^2}=-\frac{\overline{t}y}{1-s},\\
w&=-\frac{tz}{|z|^2}=-\frac{t^2 y}{s(1-s)},\\
0=\overline{x}z+\overline{y}w&=\left(-\frac{t\overline{y}}{1-s}\right)\left(\frac{ty}{1-s}\right)+\overline{y}\left(-\frac{t^2y}{s(1-s)}\right)\\
&=-t^2\left(\frac{1}{1-s}+\frac{1}{s}\right)=-\frac{t^2}{s(1-s)},
\end{align*}
in contradiction with $t\neq 0$.
\end{itemize}
In conclusion, for this class of curvature tensors, the only solutions are for
\begin{align*}
a=\sqrt{2},&& b=2, && \sigma=\frac{3}{2}e^{i\alpha}(\theta^1)^2\theta^2,&& \textrm{for }\alpha\in\smooth{M}\textrm{, or}\\
a=2,&& b=\sqrt{2}, && \sigma=\frac{3}{2}e^{i\alpha}\theta^1(\theta^2)^2,&& \textrm{for }\alpha\in\smooth{M}.
\end{align*}
We deduce that in Table \ref{table:curvatures} there are no solutions for the cases I, II, IV, V, and the only solutions in case III are the ones mentioned before.
Moreover, these solutions are isomorphic to one another and the isomorphism is obtained by swapping $u_1$ with $u_3$ and $u_2$ with $u_4$.
The simply connected Lie group corresponding to this case is $\Hyp_{\sqrt{2}}\times \Hyp_{2}$.

Notice that the unique abelian K\"{a}hler $4$-dimensional Lie algebra is flat, so its curvature is also of type (i), with $a=b=0$; thus it cannot be endowed with a projective special K\"{a}hler structure.
\item Let now $a>0, b\in\{0,1\}$ and $\Omega^{LC}=-a^2(\Omega_{\Pj^2_{\C}}+6bH_2)$, then
\begin{align*}
&[\eta\wedge\overline{\eta}]
=-\Omega^{LC}-\Omega_{\Pj^{n}_{\C}}
=(a^2-1)\Omega_{\Pj^{n}_{\C}}+6a^2bH_2\\
&\quad=\RE\left(\overline{\theta^1}\wedge\theta^1\otimes\begin{pmatrix}
2(1-a^2)&0\\
0&1-a^2\\
\end{pmatrix}
+\overline{\theta^1}\wedge\theta^2\otimes\begin{pmatrix}
0&0\\
1-a^2&0
\end{pmatrix}\right.\\
&\quad\quad\left.+\overline{\theta^2}\wedge\theta^1\otimes\begin{pmatrix}
0&1-a^2\\
0&0\\
\end{pmatrix}+\overline{\theta^2}\wedge\theta^2\otimes\begin{pmatrix}
1-a^2&0\\
0&2-2a^2+3a^2b\\
\end{pmatrix}\right).
\end{align*}
Therefore, we obtain the equations
\begin{align*}
\norm{v_1}^2=2-2a^2,&&\norm{v_2}^2=1-a^2,&&\norm{v_3}^2=2-2a^2+3a^2b.
\end{align*}
Giving the conditions
\begin{equation}
\begin{cases}
|x|^2+|y|^2=2-2a^2\\
|y|^2+|z|^2=1-a^2\\
|z|^2+|w|^2=2-2a^2+3a^2b
\end{cases}
\end{equation}
with solutions
\begin{equation}\label{eq:systemHxHnorms2}
\begin{cases}
|x|^2=1-a^2+s\\
|y|^2=1-a^2-s\\
|z|^2=s\\
|w|^2=2-2a^2+3a^2b-s
\end{cases}\qquad\textrm{ for } s\in [0,1-a^2].
\end{equation}
We now impose the vanishing of $\langle v_1,v_2\rangle$, $\langle v_2,v_3\rangle$, $\langle v_3,v_4\rangle$, that is \eqref{eq:systemHxHprod}.

We have four different cases:
\begin{itemize}
\item
Suppose at first that $y=z=0$, then $s=0$ and $a=1$, so \eqref{eq:systemHxHprod} is always satisfied, while \eqref{eq:systemHxHnorms2} becomes
\begin{align}
\begin{cases}
|x|^2=0\\
|y|^2=0\\
|z|^2=0\\
|w|^2=3b
\end{cases}.
\end{align}
It has solutions $(x,y,z,w)=(0,0,0,\sqrt{3b}e^{i\alpha})$ for $\alpha\in\smooth{M}$ and thus $(c_1,c_2,c_3,c_4)=(0,0,0,\frac{\sqrt{3b}}{2}e^{i\alpha})$.
In conclusion, $a=1$ and $\sigma=\frac{\sqrt{3b}}{2}e^{i\alpha}(\theta_2)^3$.
\item Suppose now that $z=0$ but $y\neq 0$, then $s=0$ and $a^2-1\neq 0$.
The system \eqref{eq:systemHxHprod} implies $x=w=0$, but then by \eqref{eq:systemHxHnorms2}, $0=|x|^2=1-a^2\neq 0$, so in this case there are no solutions.
\item Analogously, if $z\neq 0$ but $y=0$, then $s=1-a^2$ and \eqref{eq:systemHxHprod} gives $w=x=0$, so from \eqref{eq:systemHxHnorms2} we get $0=|x|^2=2-2a^2=2|z|^2\neq 0$ leaving no solutions.
\item The remaining case has $z\neq 0$ and $y\neq 0$.
In order to solve it, let us call $t:=\overline{y}z\neq 0$, then \eqref{eq:systemHxHnorms2} and \eqref{eq:systemHxHprod} give
\begin{align}
z&=\frac{ty}{|y|^2}=\frac{ty}{1-a^2-s},\\
x&=-\frac{\overline{t}y}{|y|^2}=-\frac{\overline{t}y}{1-a^2-s},\\
w&=-\frac{tz}{|z|^2}=-\frac{t^2 y}{s(1-a^2-s)},\\
0=\overline{x}z+\overline{y}w
&=\left(\frac{-t\overline{y}}{1-a^2-s}\right)\left(\frac{ty}{1-a^2-s}\right)+\overline{y}\left(\frac{-t^2y}{s(1-a^2-s)}\right)\\
&=-t^2\left(\frac{1}{1-a^2-s}+\frac{1}{s}\right)=-\frac{t^2(1-a^2)}{s(1-a^2-s)}.
\end{align}
The latter implies $a=1$, and from \eqref{eq:systemHxHnorms2}, we deduce a contradiction: $0< |y|^2=-s<0$.
\end{itemize}
In conclusion, the only solutions for this type of curvature tensors are obtained for
\begin{align*}
a=1, &&b=0, && \sigma=0,&&\textrm{or }\\
a=1, &&b=1, && \sigma=\frac{\sqrt{3}}{2}e^{i\alpha}(\theta^2)^3, && \textrm{for }\alpha\in\smooth{M}.
\end{align*}
In Table \ref{table:curvatures}, these results correspond to: VI for $a=1$ and $\sigma=\frac{\sqrt{3}}{2}e^{i\alpha}(\theta^2)^3$ for $\alpha\in\smooth{M}$; VII for $a=1$ and $\sigma=0$; VIII and IX for $a=\frac{1}{\sqrt{\delta}}$, $\delta>0$ and $\sigma=0$.
\end{enumerate}

Table \ref{table:casesLeft} summarises (up to isomorphisms) the cases satisfying the curvature condition, showing the non vanishing differentials of the coframe and the Levi-Civita connection.
\begin{table}[!ht]
\centering
\begin{tabular}[c]{|c|l|c|c|}
\hline
Case &Structure constants &	Levi-Civita connection &PSK\\
\hline
III&
\begin{tabular}{@{}l@{}}
$du^2=-\sqrt{2}u^{1,2}$\\
$du^4=-2u^{3,4}$\\
\end{tabular}
&
$\begin{psmallmatrix}
&\sqrt{2}u^2&\rvline&&\\
-\sqrt{2}u^2&&\rvline&&\\
\hline
&&\rvline&&2 u^4\\
&&\rvline&-2 u^4&
\end{psmallmatrix}$
&\checkmark\\
\hline
VI&
\begin{tabular}{@{}l@{}}
$du^1=2u^{1,2}$\\
$du^3=u^{2,3}$\\
$du^4=-2u^{1,3}-u^{2,4}$\\
\end{tabular}
&
$\begin{psmallmatrix}
0&-2u^1&\rvline&u^4&u^3\\
2u^1&0&\rvline&-u^3&u^4\\
\hline
-u^4&u^3&\rvline&0&-u^1\\
-u^3&-u^4&\rvline&u^1&0
\end{psmallmatrix}$
&\\
\hline
VII&
\begin{tabular}{@{}l@{}}
$du^1=u^{1,3}$\\
$du^2=u^{2,3}$\\
$du^4=-2u^{1,2}-2u^{3,4}$\\
\end{tabular}
&
$\begin{psmallmatrix}
0&u^4&\rvline&-u^1&u^2\\
-u^4&0&\rvline&-u^2&-u^1\\
\hline
u^1&u^3&\rvline&0&2u^4\\
-u^2&-u^4&\rvline&-2u^4&0
\end{psmallmatrix}$
&\checkmark\\
\hline
VIII&
\begin{tabular}{@{}l@{}}
$du^1=u^{1,3}+\frac{2}{\delta}u^{2,3}$\\
$du^2=-\frac{2}{\delta}u^{1,3}+u^{2,3}$\\
$du^4=-2u^{1,2}-2u^{3,4}$\\
$\delta>0$\\
\end{tabular}
&
$\begin{psmallmatrix}
0&\frac{2}{\delta}u^3+u^4&\rvline&-u^1&u^2\\
-\frac{2}{\delta}u^3-u^4&0&\rvline&-u^2&-u^1\\
\hline
u^1&u^2&\rvline&0&2u^4\\
-u^2&u^1&\rvline&-2u^4&0
\end{psmallmatrix}$
&\checkmark\\
\hline
IX&
\begin{tabular}{@{}l@{}}
$du^1=u^{1,4}-\frac{2}{\delta}u^{2,4}$\\
$du^2=\frac{2}{\delta}u^{1,4}+u^{2,4}$\\
$du^3=2u^{1,2}+2u^{3,4}$\\
$\delta>0$\\
\end{tabular}
&
$\begin{psmallmatrix}
0&-\frac{2}{\delta}u^4-u^3&\rvline&-u^2&-u^1\\
\frac{2}{\delta}u^4+u^3&0&\rvline&u^1&-u^2\\
\hline
u^2&-u^1&\rvline&0&-2u^3\\
u^1&u^2&\rvline&2u^3&0
\end{psmallmatrix}$
&\checkmark\\
\hline
\end{tabular}
\caption{Cases satisfying the curvature condition.}
\label{table:casesLeft}
\end{table}

Now we must check whether condition \ref{eq:condDiffCor} holds for the cases left.
Notice that for cases III, VII, VIII, IX, the K\"{a}hler form is exact with invariant potentials; respectively $-\frac{1}{\sqrt{2}} u^2-\frac{1}{2}u^4$, $-\frac{1}{2}u^4$, $-\frac{1}{2}u^4$, $\frac{1}{2}u^3$.
We can immediately say that cases VII, VIII, IX are all projective special K\"{a}hler because $\sigma=0$, so the differential condition is trivially satisfied.

Concerning case III, we can compute $d^{LC}\sigma$ by understanding how the Levi-Civita connection behaves on the unitary complex coframe $\theta$.
\begin{align*}
\nabla^{LC}\theta^1
&=\nabla^{LC}u^{1}+i\nabla^{LC}u^2
=-(\omega^{LC})^1_k \otimes u^k-i(\omega^{LC})^2_k\otimes u^k\\
&=-\sqrt{2}u^2\otimes u^2+i\sqrt{2}u^2\otimes u^1
=\sqrt{2}i u^2\otimes \theta^1;\\
\nabla^{LC}\theta^2
&=\nabla^{LC}u^{3}+i\nabla^{LC}u^4
=-(\omega^{LC})^3_k \otimes u^k-i(\omega^{LC})^4_k\otimes u^k\\*
&=-2u^4\otimes u^4+i2u^4\otimes u^3
=2i u^4\otimes \theta^2.
\end{align*}
Now we can compute 
\begin{align*}
\nabla^{LC}\sigma
&=\nabla^{LC} \left(\frac{3}{2} e^{i\alpha} (\theta^1)^2\theta^2\right)\\*
&=\frac{3}{2} i d\alpha \otimes e^{i\alpha} (\theta^1)^2\theta^2+3 \sqrt{2}iu^2 e^{i\alpha}(\theta^1)^2\theta^2+\frac{3}{2} 2i u^4\otimes e^{i\alpha} (\theta^1)^2\theta^2\\
&=-4i\left(-\frac{1}{4}d\alpha-\frac{1}{\sqrt{2}}u^2-\frac{1}{2}u^4\right)\otimes \sigma.
\end{align*}
If we define $\lambda:=-\frac{1}{4}d\alpha-\frac{1}{\sqrt{2}}u^2-\frac{1}{2}u^4$, we have that $d\lambda=\omega$ and $d^{LC}\sigma=-4i\lambda\wedge\sigma$.
Thanks to Corollary \ref{cor:charPSKesatte}, we have proven that also case III has a projective special K\"{a}hler structure for every choice of $\alpha\in\smooth{M}$.

Suppose that VI is projective special K\"{a}hler, than by Theorem \ref{theo:characterisationPSK}, locally we must have the differential condition \ref{theo:charpsk:DiffCond}.
Consider the unitary global complex coframe $\theta$.
\begin{align*}
\nabla^{LC}\theta^2
&=\nabla^{LC}u^{3}+i\nabla^{LC}u^4\\
&=u^4\otimes u^1-u^3\otimes u^2+u^1\otimes u^4+i(u^3\otimes u^1+u^4\otimes u^2-u^1\otimes u^3)\\
&=u^4\otimes \theta^1 +iu^3\otimes \theta^1-iu^1\otimes \theta^2
=i\overline{\theta^2}\otimes \theta^1-iu^1\otimes \theta^2.
\end{align*}
Thus,
\begin{align}
\begin{split}
\nabla^{LC}\sigma
&=\nabla^{LC}\left(\frac{\sqrt{3}}{2}e^{i\alpha}(\theta^2)^3\right)\\
&=id\alpha\otimes \frac{\sqrt{3}}{2}e^{i\alpha}(\theta^2)^3+3\frac{\sqrt{3}}{2}e^{i\alpha}(\nabla^{LC}\theta^2)(\theta^2)^2\\
&=id\alpha\otimes \sigma+3\frac{\sqrt{3}}{2}e^{i\alpha}(i\overline{\theta^2}\otimes \theta^1-iu^1\otimes \theta^2
)(\theta^2)^2\\
&=i(d\alpha-3u^1)\otimes \sigma+3i\overline{\theta^2}\otimes\frac{\sqrt{3}}{2}e^{i\alpha}\theta^1(\theta^2)^2;\\
d^{LC}\sigma
&=i(d\alpha-3u^1)\wedge \sigma+3i\overline{\theta^2}\wedge\frac{\sqrt{3}}{2}e^{i\alpha}\theta^1(\theta^2)^2.
\end{split}
\end{align}
Notice that this is never of the form required by condition \ref{theo:charpsk:DiffCond} for any available choice of $\sigma$, since evaluating the last component at $\theta_1$, we obtain $i\frac{\sqrt{3}}{2}\overline{\theta^2}\wedge\theta^2\otimes \theta^2$, whereas the same operation on a form of type $i\tau\wedge\sigma$ would evaluate to zero.
We deduce that VI does not admit a projective special K\"{a}hler structure.

We are now left with cases III, VII, VIII, IX.
At the level of Lie groups, case III corresponds to the connected simply connected Lie group $\Hyp_{\sqrt{2}}\times\Hyp_2$ with $\sigma=\frac{3}{2}(\theta^1)^2\theta^2$ up to isomorphism.
The other deviances are in fact obtained by taking $e^{i\alpha}\sigma$ and thus we are in the situation noted in Remark \ref{rmk:dR0_unique_structure}.
The Lie groups corresponding to the cases VII, VIII and IX, are in particular homogeneous, and they all have zero deviance, so by Proposition \ref{prop:uniquenessHyp} we deduce that they are all isomorphic to $\Hyp_{\C}^2$ as projective special K\"{a}hler manifolds.
\end{proof}
\begin{rmk}
It is striking that in case III, which is obtained via the r-map from the polynomial $x^2y$, the deviance is a global tensor which is a multiple of this polynomial with respect to a K\"{a}hler holomorphic coframe.
\end{rmk}

It turns out that all $4$-dimensional projective special K\"{a}hler Lie groups are simply connected, so this theorem already presents all possible cases.
\begin{prop}\label{prop:universalCover}
Let $(\pi\colon \widetilde{M}\to M,\nabla)$ be a projective special K\"{a}hler manifold, then the universal cover $p\colon U\to M$ admits a projective special K\"{a}hler structure.
In particular, if $\gamma\colon S\to \sharp_2 S_{3,0}M$ is the intrinsic deviance for $M$, then $p^*S\to U$ is an $S^1$-bundle and if we call $p'$ the canonical map $p^*S\to S$, then $U$ has deviance $p^*\circ \gamma\circ p'\colon p^*S\to \sharp_2 S_{3,0}U$ on $U$.

If $M$ is a projective special K\"{a}hler Lie group, then so is $U$.
\end{prop}
\begin{proof}
Since $p\colon U\to M$ is a cover, we can lift the whole K\"{a}hler structure of $M$ to $U$ by pullback $(U,p^*g,p^*I,p^*\omega)$ (the pullback of $I$ makes sense, since $p$ is a local diffeomorphism).
We will now use Theorem \ref{theo:characterisationPSK}.
The $S^1$-bundle $S$ lifts to an $S^{1}$-bundle $\pi_{p^*S}\colon p^*S\to U$, where the right action can be defined locally, since $p$ is a local diffeomorphism.
The principal connection $\varphi$ on $S$ lifts to $\varphi'=p'^*\varphi$ and its curvature is, as expected, $d\varphi'=p'^*d\varphi=-2p'\pi_S^*\omega=-2\pi_{p^*S}^*p^*\omega$.
Let $\gamma'=p^*\circ \gamma\circ p'\colon p^*S\to \sharp_2 S_{3,0}U$, then $\gamma'(ua)=a^2\gamma'(u)$ holds, as the action is defined on the fibres, which are preserved by the pullback.
The remaining properties also follow from the fact $p$ is a local diffeomorphism.

Finally, if $M$ is a Lie group with left invariant K\"{a}hler structure, then $U$ is a Lie group and its K\"{a}hler structure is also left invariant.
\end{proof}

Given a universal cover $p\colon U\to M$ of a projective special K\"{a}hler Lie group, $\ker(p)$ is a discrete subgroup and when $M$ is connected, $\ker(p)$ is in the centre $Z(U)$ of $U$.

From this observation we obtain the following corollary
\begin{cor}
A connected $4$-dimensional projective special K\"{a}hler Lie group is isomorphic to one of the following:
\begin{itemize}
\item $\Hyp_{\sqrt{2}}\times \Hyp_2$ with deviance $\sharp_2(\frac{3}{2}(\theta^1)^2\theta^2)$ in the standard complex unitary coframe $\theta$;
\item complex hyperbolic $2$-space with zero deviance.
\end{itemize}
\end{cor}
\begin{proof}
The proof follows from Theorem \ref{theo:classificazionePSK4} with Proposition \ref{prop:universalCover}, as a connected group $M$ with universal cover $p\colon U\to M$ is isomorphic to $U/\ker(p)$ and, if $M$ is a projective special K\"{a}hler Lie group, so is $U$ by Proposition \ref{prop:universalCover}.
Since $U$ is also simply connected, Theorem \ref{theo:classificazionePSK4} provides all the possibilities up to isomorphisms preserving the Lie structure.
The statement follows from the fact that these possibilities for $U$ have trivial centre.
\end{proof}
\section*{}
\bibliography{Bibliography}

\begin{thebibliography}{10}

\bibitem{CortesClassHomoSS}
D.~V. Alekseevsky and V.~Cort\'{e}s.
\newblock Classification of stationary compact homogeneous special
  pseudo-{K}\"{a}hler manifolds of semisimple groups.
\newblock {\em Proc. London Math. Soc. (3)}, 81(1):211--230, 2000.

\bibitem{SpComMan}
D.~V. Alekseevsky, V.~Cort\'{e}s, and C.~Devchand.
\newblock Special complex manifolds.
\newblock {\em J. Geom. Phys.}, 42(1-2):85--105, 2002.

\bibitem{CortesHKQK}
D.~V. Alekseevsky, V.~Cort\'{e}s, M.~Dyckmanns, and T.~Mohaupt.
\newblock Quaternionic {K}\"{a}hler metrics associated with special
  {K}\"{a}hler manifolds.
\newblock {\em J. Geom. Phys.}, 92:271--287, 2015.

\bibitem{Conification2013}
D.~V. Alekseevsky, V.~Cort\'{e}s, and T.~Mohaupt.
\newblock Conification of {K}\"{a}hler and hyper-{K}\"{a}hler manifolds.
\newblock {\em Comm. Math. Phys.}, 324(2):637--655, 2013.

\bibitem{BauesCortes}
O.~Baues and V.~Cort\'{e}s.
\newblock Abelian simply transitive affine groups of symplectic type.
\newblock {\em Ann. Inst. Fourier (Grenoble)}, 52(6):1729--1751, 2002.

\bibitem{BC2003}
O.~Baues and V.~Cort\'{e}s.
\newblock Proper affine hyperspheres which fiber over projective special
  {K}\"{a}hler manifolds.
\newblock {\em Asian J. Math.}, 7(1):115--132, 2003.

\bibitem{Berger}
M.~Berger.
\newblock Sur les groupes d'holonomie homog\`ene des vari\'{e}t\'{e}s \`a
  connexion affine et des vari\'{e}t\'{e}s riemanniennes.
\newblock {\em Bull. Soc. Math. France}, 83:279--330, 1955.

\bibitem{CandelasOssa1991}
P.~Candelas and X.~C. de~la Ossa.
\newblock Moduli space of {C}alabi-{Y}au manifolds.
\newblock {\em Nuclear Phys. B}, 355(2):455--481, 1991.

\bibitem{CFG1989}
S.~Cecotti, S.~Ferrara, and L.~Girardello.
\newblock Geometry of type {II} superstrings and the moduli of superconformal
  field theories.
\newblock {\em Internat. J. Modern Phys. A}, 4(10):2475--2529, 1989.

\bibitem{Chevalley}
C.~Chevalley.
\newblock On the topological structure of solvable groups.
\newblock {\em Ann. of Math. (2)}, 42:668--675, 1941.

\bibitem{Chu}
B.~Y. Chu.
\newblock Symplectic homogeneous spaces.
\newblock {\em Trans. Amer. Math. Soc.}, 197:145--159, 1974.

\bibitem{CortesHK1998}
V.~Cort\'{e}s.
\newblock On hyper-{K}\"{a}hler manifolds associated to {L}agrangian
  {K}\"{a}hler submanifolds of {$T^*{\bf C}^n$}.
\newblock {\em Trans. Amer. Math. Soc.}, 350(8):3193--3205, 1998.

\bibitem{CortesClassPSKr}
V.~Cort\'{e}s, M.~Dyckmanns, and D.~Lindemann.
\newblock Classification of complete projective special real surfaces.
\newblock {\em Proc. Lond. Math. Soc. (3)}, 109(2):423--445, 2014.

\bibitem{CortesCompProj}
V.~Cort\'{e}s, M.~Dyckmanns, and S.~Suhr.
\newblock Completeness of projective special {K}\"{a}hler and quaternionic
  {K}\"{a}hler manifolds.
\newblock In {\em Special metrics and group actions in geometry}, volume~23 of
  {\em Springer INdAM Ser.}, pages 81--106. Springer, Cham, 2017.

\bibitem{CiSC}
V.~Cort\'{e}s, X.~Han, and T.~Mohaupt.
\newblock Completeness in supergravity constructions.
\newblock {\em Comm. Math. Phys.}, 311(1):191--213, 2012.

\bibitem{CortesCohomogeneityOne}
V.~Cortés, M.~Dyckmanns, M.~Jüngling, and D.~Lindemann.
\newblock A class of cubic hypersurfaces and quaternionic {K}ähler manifolds
  of co-homogeneity one.
\newblock arXiv:1701.7882, 2018.

\bibitem{DorfmeisterNakajima}
J.~Dorfmeister and K.~Nakajima.
\newblock The fundamental conjecture for homogeneous {K}\"{a}hler manifolds.
\newblock {\em Acta Math.}, 161(1-2):23--70, 1988.

\bibitem{FerSab}
S.~Ferrara and S.~Sabharwal.
\newblock Quaternionic manifolds for type {${\rm II}$} superstring vacua of
  {C}alabi-{Y}au spaces.
\newblock {\em Nuclear Phys. B}, 332(2):317--332, 1990.

\bibitem{Fre1995}
P.~Fr\'{e}.
\newblock Lectures on special {K}\"{a}hler geometry and electric-magnetic
  duality rotations.
\newblock {\em Nuclear Phys. B Proc. Suppl.}, 45BC:59--114, 1996.

\bibitem{Freed1999}
D.~S. Freed.
\newblock Special {K}\"{a}hler manifolds.
\newblock {\em Comm. Math. Phys.}, 203(1):31--52, 1999.

\bibitem{Goldman}
W.~M. Goldman.
\newblock {\em Complex hyperbolic geometry}.
\newblock Oxford Mathematical Monographs. The Clarendon Press, Oxford
  University Press, New York, 1999.
\newblock Oxford Science Publications.

\bibitem{Helgason}
S.~Helgason.
\newblock {\em Differential geometry, {L}ie groups, and symmetric spaces},
  volume~34 of {\em Graduate Studies in Mathematics}.
\newblock American Mathematical Society, Providence, RI, 2001.
\newblock Corrected reprint of the 1978 original.

\bibitem{HitHKQK}
N.~Hitchin.
\newblock On the hyperk\"{a}hler/quaternion {K}\"{a}hler correspondence.
\newblock {\em Comm. Math. Phys.}, 324(1):77--106, 2013.

\bibitem{KN}
S.~Kobayashi and K.~Nomizu.
\newblock {\em {F}oundations of {D}ifferential {G}eometry}.
\newblock Intersciences Publishers, New York. Springer-Verlag, 1963,1969.

\bibitem{Swann2015}
O.~Macia and A.~Swann.
\newblock Twist geometry of the c-map.
\newblock {\em Comm. Math. Phys.}, 336(3):1329--1357, 2015.

\bibitem{MaciaSwann2019}
O.~Macia and A.~Swann.
\newblock The c-map on groups.
\newblock {\em Classical Quantum Gravity}, 37(1):015015, 17, 2020.

\bibitem{PhDThesis}
M.~Mantegazza.
\newblock {\em An intrinsic approach to the c-map}.
\newblock PhD thesis, Joint PhD program in mathematics, Università degli studi
  di Pavia, Università degli studi di Milano-Bicocca, Indam, 12 2019.

\bibitem{Ovando2004}
G.~Ovando.
\newblock Complex, symplectic and {K}\"{a}hler structures on four dimensional
  {L}ie groups.
\newblock {\em Rev. Un. Mat. Argentina}, 45(2):55--67 (2005), 2004.

\bibitem{RedBook}
S.~Salamon.
\newblock {\em Riemannian geometry and holonomy groups}, volume 201 of {\em
  Pitman Research Notes in Mathematics Series}.
\newblock Longman Scientific \& Technical, Harlow; copublished in the United
  States with John Wiley \& Sons, Inc., New York, 1989.

\bibitem{ShimaV}
H.~Shima.
\newblock Vanishing theorems for compact {H}essian manifolds.
\newblock {\em Ann. Inst. Fourier (Grenoble)}, 36(3):183--205, 1986.

\bibitem{ShimaCSC}
H.~Shima.
\newblock Hessian manifolds of constant {H}essian sectional curvature.
\newblock {\em J. Math. Soc. Japan}, 47(4):735--753, 1995.

\bibitem{TuDG}
L.~W. Tu.
\newblock {\em Differential geometry}, volume 275 of {\em Graduate Texts in
  Mathematics}.
\newblock Springer, Cham, 2017.
\newblock Connections, curvature, and characteristic classes.

\end{thebibliography}
\bibliographystyle{plain}
\end{document}